\documentclass[12pt,a4paper]{article}
\usepackage{amssymb}
\usepackage{amsmath,amsfonts,amsthm}

\usepackage{graphicx}

\usepackage{amsmath}

\usepackage{amsfonts}
\usepackage{amsthm,amscd}
\usepackage{amsmath}
\usepackage{amssymb}
\usepackage{amstext}
\usepackage{setspace}
\usepackage{float}
\usepackage{amsfonts, eufrak}

\newtheorem{theorem}{Theorem}
\newtheorem{corollary}[theorem]{Corollary}
\newtheorem{definition}[theorem]{Definition}
\newtheorem{example}[theorem]{Example}
\newtheorem{lemma}[theorem]{Lemma}
\newtheorem{notation}[theorem]{Notation}
\newtheorem{proposition}[theorem]{Proposition}
\newtheorem{remark}[theorem]{Remark}

\newcommand{\mub}{\mu_{\beta}}
\newcommand{\mubn}{\mu_{\beta,n}}
\newcommand{\phib}{\Phi_{\beta}}
\newcommand{\tr}{\text{{\normalfont Tr}}\,}

\newcommand{\4}{\vspace{.36cm}}
\newcommand{\2}{\vspace{.18cm}}

\newcommand{\hspc}{\hspace{.5cm}}

\newcommand{\N}{\mathbb{N}}
\newcommand{\R}{\mathbb{R}}

\title{Quantum Spin probabilities at positive temperature are H\"older  Gibbs probabilities}

\author{Jader E. Brasil, Artur O. Lopes,\\
Jairo K. Mengue and Carlos G. Moreira (*)  \\
\\ UFRGS, Brazil and \\
(*) IMPA - Brazil}

\begin{document}

\maketitle

\begin{abstract}

We consider the KMS state associated to  the Hamiltonian $H= \sigma^x \otimes \sigma^x$ over the quantum spin lattice
$\mathbb{C}^2 \otimes \mathbb{C}^2 \otimes \mathbb{C}^2  \otimes  ...$. For a fixed observable of the form $L \otimes L \otimes L  \otimes  ...$, where  $L:\mathbb{C}^2 \to \mathbb{C}^2 $ is self adjoint, and for  positive temperature $T$ one can get a naturally defined stationary probability $\mu_T$ on
the Bernoulli space $\{1,2\}^\mathbb{N}$. The Jacobian of $\mu_T$ can be expressed via a certain  continued fraction expansion. We will show that this probability is a Gibbs probability for a H\"older potential. Therefore, this probability is mixing for the shift map. For such probability $\mu_T$ we will show the explicit deviation function for a certain class of  functions.
When decreasing temperature we will be able to exhibit the explicit transition value $T_c$ where the set of values of the Jacobian of the  Gibbs probability $\mu_T$ changes from being a Cantor set to being an interval.

 We also present some  properties for quantum spin probabilities at zero temperature (for instance, the explicit value of the entropy).

\end{abstract}

\section{Introduction}

In \cite{LMMM} the case of zero temperature for the quantum spin probability and the Hamiltonian $H= \sigma^x \otimes \sigma^x$, where $\sigma_x$ is the $x$-Pauli matrix, was analyzed. Here we will analyze the analogous problem in the case of positive temperature.

Given a selfadjoint operator $H$ acting on a finite dimensional complex Hilbert space $\mathcal{H}$ and a temperature $T>0$,  the density operator
$$\rho_{H,T}= \frac{e^{-\frac{1}{T}\, H }}{Z(T)}, $$
where $Z(T)=$ Trace $e^{-\frac{1}{T}\, H }$, is called the KMS operator associated to the Hamiltonian $H$. It is usual to denote $\beta=1/T$. A general reference on KMS operators and KMS states is \cite{Bra}.

%We will also fix later a certain  self adjoint operator $L : \mathbb{C}^2 \to  \mathbb{C}^2.$

The set of linear operators acting on $\mathbb{C}^2$ will be denoted by $\mathcal{M}_2.$
We will call
$\omega=\omega_n:  \underbrace{\mathcal{M}_2 \otimes \mathcal{M}_2 \otimes...
\otimes \mathcal{M}_2}_n\to \mathbb{C}$ a $C^*$-dynamical state if $\omega_n ( I^{ \otimes \, n})=1$ and $\omega_n(a)\geq 0$, if $a$ is a non-negative element in the tensor product. General references on tensor products and spin lattices are \cite{Evans}, \cite{Ara1} and \cite{PF}.

\medskip

Here we will consider the Hamiltonian $H=\sigma^x \otimes \sigma^x: \mathbb{C}^2 \otimes  \mathbb{C}^2\,\to \mathbb{C}^2 \otimes  \mathbb{C}^2$ acting on the spin lattice $\mathbb{C}^2 \otimes \mathbb{C}^2 \otimes \mathbb{C}^2  \otimes  ...$, where
$$\sigma_x=\left(
\begin{array}{cc}
0 & 1\\
1 & 0
\end{array}\right)$$
is the $x$-Pauli matrix. More precisely we will consider for each $n\geq 2 $
the state  $\omega=\omega_n$, which is  defined in the following way: consider a fixed value $\beta>0$,
and $ H_n= (\mathbb{C}^2)^{\otimes n} \to (\mathbb{C}^2)^{\otimes n}$ given by
$$H_n=\sum_{j=0}^{n-2} I^{ \otimes \,j} \otimes  H \otimes  I^{ \otimes\,(n-j-2) }.  $$
Denote $\rho_{\omega}$ the operator
$$ \rho_{\omega}=\rho_{ \omega_{H,\beta,n}} = \frac{1}{Tr\, ( e^{-\, \beta H_n}  ) }  e^{-\, \beta H_n}$$
and define the $C^*$-dynamical state $\omega_n=\omega_{H,\beta,n}$ by
$$ \omega_n\,(A)\, = \text{ trace }\,\, (\rho_{\omega} \,A),\,\,\,\,A \in   \underbrace{\mathcal{M}_2 \otimes \mathcal{M}_2 \otimes...
	\otimes \mathcal{M}_2}_n. $$

The family $w_n$, $ n \in \mathbb{N}$, defines a  $C^*$-dynamical state over $\mathcal{M}_2^{ \otimes\, \infty}$ in the following sense:
$w_n$ is a  $C^*$-dynamical state over
$\underbrace{\mathcal{M}_2 \otimes \mathcal{M}_2 \otimes... \otimes \mathcal{M}_2}_n$ for each $n$.
The $C^*$-dynamical states play an important role in Quantum Statistical Mechanics (see \cite{Bra} and \cite{GS})

\medskip

{\bf Assumption A:}   We fix a value $\theta \in (0, \pi/2),\, \theta \neq \pi/4,$ and  we consider the self-adjoint operator $L$ on the form

$$L=\, \left(
\begin{array}{cc}
\cos^2(\theta)-\sin^2(\theta) & 2\cos(\theta)\sin(\theta)\\
2\cos(\theta)\sin(\theta) & \sin^2(\theta)-\cos^2(\theta)
\end{array}
\right).$$

The case $\theta =0$ corresponds to $L=\sigma^z$ and the case $\theta = \pi/2$ corresponds to $L=-\sigma^z$. We will not consider these cases.

\medskip

What is important in the above choice of $L$  is the corresponding subspaces of eigenvectors.
The eigenvalues of $L$ are $\lambda_1=1$ and $\lambda_2=-1$ which are associated, respectively, to the unitary eigenvectors $\psi_1= (\cos(\theta), \sin(\theta))\in \mathbb{C}^2 $ and $\psi_{2} =(-\sin(\theta),\cos(\theta))\in\mathbb{C}^2$, which are orthogonal. Furthermore, for any $n\in \mathbb{N}$, the observable
$$ L^{ \otimes \, n}:=(L \otimes L \otimes \,....\otimes L): (\mathbb{C}^d \otimes \mathbb{C}^d \otimes \,....\otimes \mathbb{C}^d)
\to (\mathbb{C}^d \otimes \mathbb{C}^d \otimes \,....\otimes \mathbb{C}^d)$$
has the eigenvector $ (\, \psi_{j_1} \otimes \psi_{j_2} \otimes .... \otimes  \psi_{j_n}\,)$
associated to the eigenvalue  $\lambda_{j_1} \cdot\lambda_{j_2} \cdots  \lambda_{j_n}.$ Any eigenvalue of $ L^{ \otimes \, n}$ is of this form.

\medskip

We denote by  $P_j=P_{\psi_j}: \mathbb{C}^2 \to \mathbb{C}^2$ the orthogonal projection on the subspace
	generated by $\psi_j$, \,$j\in\{1,2\}$. In this way
			
			$$
			P_1= \, \left(
			\begin{array}{cc}
			\cos(\theta)^2 & \cos(\theta)\, \sin(\theta)\\
			\cos(\theta)\, \sin(\theta) & \sin(\theta)^2
			\end{array}
			\right)$$
			
			\noindent and
			
			$$
			P_2=  \left(
			\begin{array}{cc}
			\sin(\theta)^2  & -\cos(\theta)\, \sin(\theta)\\
			- \cos(\theta)\, \sin(\theta) & \cos(\theta)^2
			\end{array}
			\right) .$$

			\4 Note that Tr $P_1=$  Tr $P_2= 1$. Moreover,
			$$\sigma^x (P_1) =\left(
			\begin{array}{cc}
			\cos(\theta)\sin(\theta) & \sin^2(\theta)\\
			\cos^2(\theta) & \cos(\theta)\sin(\theta)
			\end{array}
			\right),
			$$
			
			\noindent which has trace equal to $\beta_1:= \sin(2\, \theta)\in \mathbb{R} $, and
			
			$$\sigma^x (P_2) =
			\, \left(
			\begin{array}{cc}
			-\cos(\theta)\sin(\theta) & \cos^2(\theta) \\
			\sin^2(\theta)& -\cos(\theta)\sin(\theta)
			\end{array}
			\right)$$
			has trace $  \beta_2:=- \sin(2\, \theta)\in \mathbb{R}$. Therefore,
			$\tr(\sigma^x (P_2))=\beta_2= - \beta_1$.

			\medskip

			We want to define  a probability $\mub$ on the Bernoulli space $\{1,2\}^\mathbb{N}.$ First, for each $n$ we introduce the probability  $\mubn$, in such way that, for an element $(j_1,...,j_n) \in \{1,2\}^n$ it is given by
			
			$$\mubn( j_1,,...,j_n )=\frac{1}{Tr\, ( e^{-\, \beta H_n}  ) }\tr \left[\,e^{-\, \beta H_n} ( P_{{j_1}} \otimes  P_{{j_2}}\otimes...\otimes  P_{{j_n}})\right]=$$
			$$\frac{1}{\cos^{n-1}(i\beta)2^n}\tr\left[\prod_{i=1}^{n-1} [\cos(i \beta)  I^{\otimes n} +i \sin(i \beta)(\sigma^x_{i}\otimes \sigma^x_{i+1})_n](   P_{j_1} \otimes...\otimes P_{j_n})\right],$$
			where the above product represents composition of operators.
			
Finally we observe that there exists a unique probability
$\mub$ over $\Omega=\{1,2\}^{\mathbb{N}}$ satisfying
$\mub([j_1,...,j_n]) = \mubn(j_1,...,j_n)$ for any $n \in \{2,3,...\}$ and any cylinder set  $[j_1,...,j_n]\subset \Omega$. It is also invariant for the shift map $\sigma$ (see Theorem \ref{def_mub} below).

\begin{definition}
We call the above probability $\mu_\beta$ the quantum spin probability for inverse temperature $\beta$.
\end{definition}

The above definition is consistent with the one usually considered on the literature (see \cite{Hi}, \cite{Leb} and \cite{Len}).

\medskip

Suppose $J:\{1,2\}^\mathbb{N}\to\mathbb{R}$ is a H\"older positive function such that for any $y\in \{1,2\}^\mathbb{N}$ we have that
$\sum_{\{x\,|\, \sigma(x)=y\}} J(x)=1$.
The associated  Ruelle  operator $\mathcal{L}_{\log J}$  is the one such that $\mathcal{L}_{\log J}(f)=g$, when for any $y$ we have
$$ g(y) = \sum_{\{x\,|\, \sigma(x)=y\}} J(x)\, f(x).$$
$\mathcal{L}_{\log J}^*$ denotes the dual of the Ruelle operator which acts on probabilities over $\{1,2\}^\mathbb{N}$
(via the Riesz Theorem) as described in \cite{PP}.

\begin{definition} The unique probability $m$ such that $\mathcal{L}_{\log J}^*(m)=m$ is called the H\"older Gibbs probability associated to $\log J$. We say that $J$ is the Jacobian of $m$.
\end{definition}

One can show that the Kolomogorov entropy satisfies $h(m) = - \int \log J d m.$ The function $J$ can also be seen as the Radon-Nykodim derivative on the inverse branches of $\sigma$ (see section 9.7 of \cite{Viana}  or \cite{Pa}).
H\"older Gibbs probabilities are equilibrium states for H\"older potentials (see \cite{PP}). The relation of Gibbs probabilities with DLR probabilities is explained in \cite{CL1} and \cite{CL2}.

Our main result is the following (see section \ref{Ho}):

\begin{theorem} \label{oks} For any $\beta>0$ there exist a Holder function $J_\beta$, such that, $J_\beta$ is the Jacobian of the quantum spin probability $\mu_\beta$. The function $J_\beta$ is described by a continuous fraction expansion (see Lemma \ref{tri}). There is an explicit  transition parameter $\beta_c$ (when  $\beta$ is increasing) where  the image values of the Jacobian change from  a regular Cantor set to an interval.

\end{theorem}

\medskip

We will present later on section \ref{Ho} some pictures illustrating this behaviour (Cantor set or interval).

\medskip

 We say that there exists a {\bf Large Deviation Principle} (LDP for short) for the probability $\mu$ on $\Omega$ and the function $A: \Omega \to \mathbb{R}$, if there exist a lower-semicontinuous function $I:\mathbb{R} \to \mathbb{R}$, such that,
 \medskip

 a) for all closed sets $K \subset \mathbb{R}$ we get
$$ \lim_{n \to \infty} \frac{1}{n} \log  \bigg(\mu\,\,\bigg\{ z  \,\text{such that}\,\,\,\, {1 \over n} \sum_{j=0}^{n-1} A(\sigma ^{j} (z)) \in K\, \bigg\}\,\, \bigg) \leq \, - \inf_{s \in K} \, I(s),$$

b) for all open  sets $B \subset \mathbb{R}$ we get
$$ \lim_{n \to \infty} \frac{1}{n} \log  \bigg(\mu\,\,\bigg\{ z  \,\text{such that}\,\,\,\, {1 \over n} \sum_{j=0}^{n-1} A(\sigma ^{j} (z)) \in B\, \bigg\}\,\, \bigg)\geq \,- \inf_{s \in B} \, I(s).$$

\medskip

The above function $I$ is called the deviation function. We refer the reader to \cite{Leb}, \cite{Len}, \cite{Hi} and \cite{Oga} for several results on the topic of Large Deviations for Quantum Spin Systems.
It is known that when the probability $\mu$ is H\"older Gibbs (the case we consider here) and the function $A$ is Holder then it is true the Large Deviation principle and $I$ is an analytic function.

\medskip

Here we will present  explicit results for a certain potential $A:\Omega \to \mathbb{R}$ which depends just on the first coordinate, that is,
$$A(x_1,x_2,...,x_n...)=A(x_1).$$
Such $A$ is a H\"older function.

Here we denote	$\phib = \dfrac{e^{-\beta}}{\cos(i \beta)}= \frac{2}{e^{2\beta} + 1} \in (0,1)$, $\beta_1= \sin(2\, \theta)$ and $\beta_2=-\beta_1$. We assume that $A$ depends just on the first coordinate and  we set $$\displaystyle\delta(t)= \sum_{j}  e^{t\, \, A(j)}\,\text{ and }\,\alpha(t)=  \sum_{j} \beta_{j}\, e^{t\, \, A(j)}.$$

 We will prove the following result:

\begin{theorem}\label{teo2} Denote $\mu_\beta$,  $\beta>0$ the quantum spin probability.
In the case $A:\Omega \to \mathbb{R}$ depends just on the first coordinate on $\Omega$, the associated free energy function
$$c(t)=  \lim_{n \to \infty}  \frac{1}{n}\,\, \log \int e^{t\,(\,A(x) + A(\sigma(x)) + A(\sigma^{2} (x)) + ...
	+ A(\sigma^{n-1} (x))\, \,)} d \mu_\beta (x)$$
is given  by the expression

\begin{align*}
			&c(t)=\log\left( \frac{ \phib\delta(t) + \sqrt{\phib^2\delta(t)^2 + 4(-1+\phib)(\alpha(t)^2-\delta(t)^2)}}{4}\right).
			\end{align*}
The deviation function $I$ is the Legendre Transform of $c(t)$.
\end{theorem}

\bigskip

The last claim follows from classical results.
As the probability $\mu_\beta$ is a H\"older Gibbs state the deviation function $I$ is the  Legendre transform of  $c(t)$ (see \cite{El}, \cite{Lo2}, \cite{Lo7} or \cite{Lo3}). The proof of the above Theorem will be done on section \ref{LDP}.

\medskip

In section \ref{zero} we will present some results for the quantum spin probability $\mu$ at temperature zero which complement the ones in \cite{LMMM}. Among other things we compute the entropy of such $\mu$ and we present an ergodic conjugacy with another dynamical system which is somehow ``related'' to the independent Bernoulli probability.

\medskip

   \begin{theorem} \label{zeze} Denote by $\mu$ the zero temperature quantum spin probability, as described in \cite{LMMM}, by  $\mu_0$ the uniform probability $(1/2,1/2)$ on the set $\{0,1\}$ and by $\mu_p$ the independent Bernoulli probability $(p,1-p)$ on $\{0,1\}^{\mathbb N}$, where $p=\frac{1+\sin(2\, \theta)}2 .$
   There exists an ergodic equivalence $H$ between the shift $\sigma$
   	acting on  $(\{1,2\}^{\mathbb N},\mu)$ and the  transformation $T$ acting on $(\{0,1\} \times \{0,1\}^{\mathbb N},\mu_0\times \mu_p)$, where
   $$T(c_0,c_1,...)=(1-c_0,\sigma(c_1,c_2,...)).$$

   The entropy of $\mu$ is $h(\mu) \,=\, - p\, \log p - (1-p)\, \log (1-p).$

   \end{theorem}
	
\medskip

We will present later on section \ref{zero} a picture showing the behaviour of the values of the Jacobian of the zero temperature quantum spin probability in this case.

Part of the present paper was described on the Master dissertation \cite{Br}.

\medskip

In the appendix we will provide some proofs (of theorems and propositions of the paper) which are more technical.
		
		%	Note que a expressão explicita de $c(t)$ permite estimar (via transformada de Legendre) a função de desvio $I$.
		%	\medskip

\section{Recurrence formulas and construction of the quantum spin probability}

In this section we will study initial properties of the quantum spin probability and provide recurrence formulas for calculate it in cylinders.

%Note that $(\sigma^x)^2=I.$

\begin{notation} For $n\ge 3$,
				\[(\sigma^x_{l}\otimes \sigma^x_{l+1})_n =\left\{	
				\begin{array}{ll}
				\sigma^x\otimes \sigma^x \otimes \underbrace{I  \otimes... \otimes I}_{n-2},& l=1\\
				\underbrace{I  \otimes... \otimes I}_{n-2}\otimes \sigma^x\otimes \sigma^x,& l=n-1\\
				\underbrace{I  \otimes... \otimes I}_{l-1}\otimes \sigma^x\otimes \sigma^x \otimes \underbrace{I  \otimes... \otimes I}_{n-l-1}, & 1<l<n-1		\end{array}\right..
	\]
			\end{notation}
%			 Now, we will estimate $e^{-\beta H_n}:(\mathbb{C}^2)^{ \otimes n} \to (\mathbb{C}^2)^{\otimes n}$. As $(\sigma^x_{i}\otimes \sigma^x_{i+1})_n$ commutes with $(\sigma^x_{j}\otimes \sigma^x_{j+1})_n$, $i,j \in \{1,...,n-1\}$, we get
From \cite{LMMM} we get
			\begin{align*}
			e^{-\beta H_n} &= e^{-\, \beta\,[\sum_{l=1}^{n-1} (\sigma^x_{l}\otimes \sigma^x_{l+1})_n]}=\prod_{l=1}^{n-1}e^{ -\, \beta\,(\sigma^x_{l}\otimes \sigma^x_{l+1})_n}\\
			&= \prod_{l=1}^{n-1} [\cos(i \beta) \, I^{\otimes n} \,+\,i\, \sin(i \beta)\, (\sigma^x_{l}\otimes \sigma^x_{l+1})_n],
			\end{align*}
			\noindent where the product means composition of operators. Furthermore
 ${\tr(e^{-\beta\,H_n})}={\cos^{n-1}(i \beta) 2^n}$.

		We 	 denote $\Phi= \phib := \dfrac{e^{-\beta}}{\cos(i \beta)}$. Note that as $i\,\sin(\beta i)= - \cos(\beta i)+e^{-\beta}$, we get $\dfrac{i\,\sin(i\beta)}{\cos(i\beta)} = -\,1 \,+\, \phib$.
	 In this way  we can express the probability $\mubn$ as: 	{\small 		
				\begin{equation}\label{def_mubn} \mubn(j_1,...,j_n) =\frac{1}{2^n}\tr\left[\prod_{i=1}^{n-1} [I^{\otimes n} + ( \phib -1) (\sigma^x_{i}\otimes \sigma^x_{i+1})_n](   P_{j_1} \otimes ...\otimes P_{j_n})\right].
				\end{equation}	}

		%			\begin{equation}\label{def_mubn}
		%			\hspace{-10px}\mubn(j_1,...,j_n) = \frac{1}{2^n}\tr\left[\prod_{i=1}^{n-1} [I^{\otimes n} \,-\, (\sigma^x_{i}\otimes \sigma^x_{i+1})_n \,+\, \phib\, (\sigma^x_{i}\otimes \sigma^x_{i+1})_n](   P_{j_1} \otimes  P_{j_2}\,\otimes...\otimes P_{j_n})\right]
		%			\end{equation}	

We define  $\mu_{\beta,1}(1)=1/2$ and $\mu_{\beta,1}(2)=1/2$. For cylinders of length $2$ and $3$ the probability is computed below.

			\begin{example}\label{ex_mu2emu3}

				\begin{align*} \mu_{\beta,2}(a,b) &= \frac{1}{2^2}\tr\left[ (I\otimes I \,-\, (\sigma^x_{1}\otimes \sigma^x_{2})_2 \,+\, \phib\, (\sigma^x_{1}\otimes \sigma^x_{2})_2)(   P_{a} \otimes  P_{b})\right] \\
				&= \frac{1}{2^2}\tr\left[ P_{a} \otimes  P_{b} \right] - \frac{1}{2^2}\tr\left[ \sigma^x_{1}P_{a} \otimes  \sigma^x_{2}P_{b} \right] + \frac{1}{2^2}\,\phib\,\tr\left[ \sigma^x_{1}P_{a} \otimes  \sigma^x_{2}P_{b} \right]  \\
				&= \frac{1}{2^2}[1\,-\,\beta_{a}\beta_{b}] + \frac{1}{2^2}\phib\beta_{a}\beta_{b}.
				\end{align*}
and
	\begin{align*} \mu_{\beta,3}(a,b,c) &= \frac{1}{2^3}\tr\left[\begin{array}{lll}
				(I\otimes I \otimes I \,+\, (-1+\phib)(\sigma^x_{1}\otimes \sigma^x_{2})_3) \circ\\
				(I\otimes I \otimes I \,+\, (-1+\phib)(\sigma^x_{2}\otimes \sigma^x_{3})_3)\circ\\
				(P_{a} \otimes  P_{b} \otimes  P_{c})	
				\end{array}\right]\\
				&= \frac{1}{2^3}\bigg[ 1 + (-1+\phib)\beta_a\beta_b  + (-1+\phib)\beta_b\beta_c + (-1+\phib)^2\beta_a\beta_c\bigg].
				\end{align*}

			\end{example}

\medskip

			\begin{remark}\label{phibin01}
				We have $0 < \phib < 1$, for all $\beta \in (0, \infty)$. Indeed, $e^{\beta} = \cos(i\beta) + i \sin(i\beta)$, then
%				= \frac{2e^{-\beta}}{e^{\beta}+e^{-\beta}} = \frac{2}{e^{2\beta} + 1}

				$$\cos(i\beta) = \frac{e^{\beta} + e^{-\beta}}{2}.$$
				
				\noindent In this way we get
				
				$$\phib = \frac{2e^{-\beta}}{e^{\beta}+e^{-\beta}} = \frac{2}{e^{2\beta} + 1} \in (0,1).$$		
				
			\end{remark}
			
The computations of $\mu_\beta$ for cylinders of size $4,5,...$ can be done from recurrence formulas introduced below.			
Before this, let us present the following result			
			\medskip

			\begin{theorem}\label{def_mub} There exists a unique probability
				$\mub$ over $\Omega := \{1,2\}^{\mathbb{N}}$, such that, for any $n \in \{1,2,3,...\}$ and any cylinder set  $[j_1,...,j_n]$ we have:
				$$\mub([j_1,...,j_n]) = \mubn(j_1,...,j_n).$$ Furthermore it is invariant by the shift map $\sigma$.
				
			\end{theorem}
			\begin{proof}
			From Theorem \ref{hipotese_kolmogorov1} in Appendix, for all $n\in \{1,2,3,...\}$ we get
			\begin{equation*}
			\mu_{\beta,n+1}(j_1, ..., j_n, 1) + \mu_{\beta,n+1}(j_1, ..., j_n, 2) = \mubn(j_1, ..., j_n).
			\end{equation*}
			In a similar way, from symmetry of (\ref{def_mubn}), we obtain, for all  $n\in \{1,2,3,...\}$,
			\[\mu_{\beta,n+1}(1, j_1, ..., j_n) + \mu_{\beta,n+1}(2, j_1, ..., j_n)=	\mubn(j_1, ..., j_n).\]
			The existence and uniqueness follow from the Caratheodory extension theorem
			 (see \cite{Durrett} or \cite{Viana}). The invariance by  $\sigma$ follows from  above equation.
			\end{proof}
			
			\begin{notation}
				$\mub(j_1,...,j_n) := \mub([j_1,...,j_n]).$
			\end{notation}
\medskip

			The ergodic properties of the probability  $\mub$ is the main object of the present paper. The case when temperature is zero $(\beta \to \infty$) was considered in \cite{LMMM}. In the end of the paper  we will present some more
results which complement the analysis of \cite{LMMM}.
Related results appear in \cite{RMNS} and \cite{Pil}.

\medskip

From now on let us present two recurrence formulas which are analogous - but more complex - to the ones in \cite{LMMM}.
		
		\begin{theorem}\label{medidacilindro}
			For the probability  $\mub$ and for any $n\ge 2$, we get
			$$
			\mub(k,j_1,...,j_n)= $$
			$$\frac{\mub( j_1,...,j_n )}{2}  + \sum_{i=1}^{n-2}\frac{(-1+\phib)^i \beta_k \beta_{j_i} }{2^{i+1}}\mub(j_{i+1},...,j_n)  +$$
$$\frac{(-1+\phib)^{n-1}\beta_k\beta_{j_{n-1}}}{2^{n+1}}+ \frac{(-1+\phib)^{n}\beta_k\beta_{j_n}}{2^{n+1}}.
			$$
		\end{theorem}

For the proof  see Appendix (Theorem \ref{medidacilindro1})

		\bigskip
		
		\begin{theorem}\label{recursive}
			For any $n\geq 1$, we get
			$$
				\mub (k_0 , k_{1} ,...,k_n)=$$
$$\frac{1}{2}\bigg(1+\frac{\beta_{k_0}}{\beta_{k_1}}(-1+\phib)\bigg)\mub (k_1 , ...,k_{n})\,+
										   \frac{\beta_{k_0}}{2}(-1+\phib) \bigg(\frac{-1}{2\beta_{k_{1}}}+\frac{\beta_{k_{1}}}{2}\bigg)\mub (k_2,...,k_{n}).
			$$

		\end{theorem}

\medskip
For the proof see Appendix (Theorem \ref{recursive1})

\medskip

\section{The continuous fraction expression for the Jacobian}

A general reference for Thermodynamic Formalism and Gibbs probabilities is \cite{PP}.
The reasoning of this section is similar to the one in section 4 in \cite{LMMM}.

Remember that $\phib = \frac{2}{e^{2\beta} + 1} \in (0,1).$

		 We denote
		$$ a(k_0,k_1) = \frac{1}{2}\bigg(1+\frac{\beta_{k_0}}{\beta_{k_1}}(-1+\phib)\bigg) $$

and

		$$ b(k_0,k_1) = \frac{\beta_{k_0}}{4}(-1+\phib) \bigg(\frac{-1}{\beta_{k_{1}}}+\beta_{k_{1}}\bigg),$$
		
		%e $$\gamma= \frac{1}{4} + \frac{1}{4}(-1\,+\,\phib)\beta_{1}^2 =  \mub(1,1) $$
		\noindent where $k_0,k_1\in \{1,2\}$, $\beta_1=\sin(2\theta)$ and $\beta_2=-\beta_1$ for $\theta \in (0,\frac{\pi}{2})$ and $\theta \ne \frac{\pi}{4}$ .
		 The possible values of  $a(k_0,k_1)$ e $b(k_0,k_1)$ are:

\medskip

		a) if $k_0=k_1$, then
		
		$$ 0 < a(k_0,k_1)=\dfrac{\phib}{2} < \frac{1}{2}$$
		
		\noindent and
		
		$$0 < b(k_0,k_1)=\frac{\beta_{k_0}}{4}(-1+\phib) \bigg(\frac{-1}{\beta_{k_{1}}}+\beta_{k_{1}}\bigg)= $$
		$$\frac{1}{4}(-1+\phib) (-1+\beta_{k_{1}}^2) = \frac{1}{4}(1-\phib) (1-\beta_{k_{1}}^2) < \frac{1}{4}.$$
		
\medskip

		\noindent b) if $k_0\neq k_1$, then
		$$ 0 < a(k_0,k_1)=1 - \dfrac{\phib}{2} < 1$$
		
		\noindent and
		
		$$-\frac{1}{4} < b(k_0,k_1)=\frac{\beta_{k_0}}{4}(-1+\phib) \bigg(\frac{-1}{\beta_{k_{1}}}+\beta_{k_{1}}\bigg)$$
		$$= -\frac{1}{4}(-1+\phib) (-1+\beta_{k_{1}}^2) = -\frac{1}{4}(1-\phib) (1-\beta_{k_{1}}^2) < 0.$$
		
		\medskip
		
		 By Proposition \ref{recursive} we get
		$$\mub (k_0 , k_{1} , ...,k_n)=a(k_0,k_1)\mub (k_1 , ...,k_{n}) +b(k_0,k_1)\mub (k_2,...,k_{n}).$$

		\begin{proposition}\label{positivaemcilindros}
			Suppose that $\theta \ne \frac{\pi}{4}$. Then, $\mub$ is positive in cylinders.
		\end{proposition}
		
		\begin{proof}
	The proof follows the arguments in \cite{LMMM}.	For cylinders of size 1, 2 and 3 the expression can be directly checked. We conclude the proof using induction.

Suppose $\mub$ is positive for any cylinder set of size smaller or equal to $n$. As $\mub$ is $\sigma$-invariant, we get
			$$ \mub(x_0,x_1,...,x_n) = \mub(1,x_0,...,x_n) + \mub(2,x_0,...,x_n) \geq $$
$$\mub(x_0,x_0,x_1,...,x_n)= a(x_0,x_0)\mub (x_0 , ...,x_{n}) +b(x_0,x_0)\mub (x_1,...,x_{n}) > 0.$$
			
		\end{proof}
		
		\medskip
			
		 Now we will get a result on the Jacobian $J=J_{\mu_\beta}$ (see section 9.7 of \cite{Viana}) of the probability $\mub$ in a similar fashion as in \cite{LMMM}.

Define for $x=(x_0,x_1,...)\in\Omega$,
		
		$$J^n(x) := \dfrac{\mub(x_0,...,x_n)}{\mub(x_1,...,x_n)}$$
		and
		$$ J_{\mub}(x) = J(x) := \lim\limits_{n\to\infty} J^n(x) = \lim_{n\to\infty} \dfrac{\mub(x_0,...,x_n)}{\mub(x_1,...,x_n)}  $$
		in the case the limit exists. It is known that when  $\nu$ is a $\sigma$-invariant  probability the  Jacobian $J_{\nu}$ is well defined \textit{for almost everywhere point} $x$ (see \cite{Pa}).
			
		 From proposition \ref{recursive} we get
		
		\begin{equation}\label{eq_seq_jacob}
		\dfrac{\mub (k_0 , k_{1} , ...,k_n)}{\mub (k_1 , ...,k_{n})}=a(k_0,k_1)\,+\,b(k_0,k_1)\dfrac{\mub (k_2,...,k_{n})}{\mub (k_1 , ...,k_{n})}.
		\end{equation}
				From this expression we get the following important consequence:
\medskip

		\begin{corollary}\label{rec_do_jacobiano} For all $n\geq 1$ we get
			$$ J^n(x_0,x_1,...) =  a(x_0,x_1)\,+\,b(x_0,x_1)\dfrac{1}{J^{n-1}(x_1,x_2,...)}.$$
			Moreover, taking the limit $n\to\infty$ in (\ref{eq_seq_jacob}), we obtain
			$$J(k_0,k_1,...) = a(k_0,k_1)\,+\,b(k_0,k_1) \dfrac{1}{J(k_1,k_2,...)}.$$
		\end{corollary}
			
		\begin{remark}\label{obs_jacob_n_lema}	

Note that adapting the argument of Proposition  \ref{positivaemcilindros}, we get
						$$J^n(x_0,x_1,...) \ge b(x_0,x_0) > 0. $$
						 Moreover, note that $J^n(1,x_1,...)+J^n(2,x_2,...) = 1$. Therefore, we obtain
			
			$$ b(x_0,x_0) \le J^n(x_0,x_1,...) \le 1 - b(x_0,x_0).$$
			
			 Finally, as $J(x) = \lim\limits_{n\to\infty} J^n(x)$ (if the limit exists) we get the estimate
			
			$$ \frac{1}{4}(1-\phib) (1-\beta_{x_{0}}^2) \le J(x) \le 1 - \frac{1}{4}(1-\phib) (1-\beta_{x_{0}}^2). $$
		\end{remark}
	
		\medskip
		
		\begin{lemma} \label{tri} In the case of convergence of $J$ we get
			$$
			J(k_0,k_1,...) = \lim_n \dfrac{\mub(k_0,...,k_n)}{\mub(k_1,...,k_n)} =$$
			$$ \lim_{n\to \infty} \left[ a(k_0,k_1)\,+\,b(k_0,k_1)\dfrac{1}{a(k_1,k_2)\,+\,b(k_1,k_2) \dfrac{1}{...a(k_{n-1},k_n) + b(k_{n-1},k_n)\frac{1}{1/2}}} \right]
			$$			
		\end{lemma}
		\begin{proof}
			From Corollary \ref{rec_do_jacobiano} it follows
			
			$$ J^n(x_0,x_1,...) =  a(x_0,x_1)\,+\,b(x_0,x_1)\dfrac{1}{J^{n-1}(x_1,x_2,...)}.$$
			
			Note that from Proposition $\ref{recursive}$, we get
			
			$$J^1(x_{n-1},x_n,...) = a(x_{n-1},x_n) + b(x_{n-1},x_n)\frac{1}{1/2}.$$
		
			As $J(x) = \lim\limits_{n\to\infty} J^n(x)$ the Lemma is proved.
		\end{proof}
		
The expression for $J(k_0,k_1,...)$ given by Lemma \ref{tri} is a continuous fraction expansion (when converges).
A general reference for continuous fraction expansions is \cite{Wall}. In the next section we will prove that
$J$ is well defined for any element on $\{1,2\}^\mathbb{N}$ and it is a H\"older continuous function. Therefore, we will show that $\mu_\beta$ is a Gibbs probability (see \cite{PP}). As a consequence $\mu_\beta $ is mixing (see \cite{PP} or \cite{Viana}).

\medskip

\section{$\mu_\beta$  is a H\"older Gibbs probability} \label{Ho}

In the reasoning of this section the continuous fraction expansion expression presented in  Lemma \ref{tri} will be of fundamental importance. We will show that $J: \{1,2\}^\mathbb{N} \to \mathbb{R}$ is H\"older continuous.

We will show that there exists a critical parameter $\beta_c =\frac{1}{2}  \log (\frac{4}{ \cos^2 (2\,\theta)} - 1)$, such that, for $\beta < \beta_c$ the set of values of the Jacobian is a regular Cantor set and for $\beta\geq \beta_c$
this set is an interval.

\medskip

We denote in this section %$ \Phi = \frac{2}{e^{2\beta} + 1} \in (0,1)$,
$$ \alpha= a(k_0,k_0)=\frac{ \Phi_\beta}{2}= \frac{1}{e^{2\, \beta}+1}$$
 \text{and}
$$\gamma= b(k_0,k_0)=\frac{1}{4}\, (1-  \Phi_\beta) \cos^2 (2 \, \theta). $$
In this case $\alpha$ and $ \gamma$ are such that $0< \alpha<1/2,$ and $0< \gamma<1/4$. Furthermore, since $4\gamma<1-\Phi_\beta=1-2\alpha$, $4 \gamma + 2 \alpha<1$ and so $(1-\alpha)^2 -4\gamma >0$.

We want to identify $J(k_0,k_1,k_2,...)$ as the  limit
$$ \lim_{n\to \infty} \left[ a(k_0,k_1)\,+\,b(k_0,k_1)\dfrac{1}{a(k_1,k_2)\,+\,b(k_1,k_2) \dfrac{1}{...a(k_{n-1},k_n) + b(k_{n-1},k_n)\frac{1}{1/2}}} \right].
			$$	

Note that if $k_0=k_1$ we get that
\begin{equation} \label{li1}  J(k_0,k_1,k_2,...)  = \alpha + \frac{\gamma}{J(k_1,k_2,k_3,...)} ,\end{equation}
and in the case $k_0\neq k_1$ we get that
\begin{equation} \label{li2} J(k_0,k_1,k_2,...)  = 1-  \alpha - \frac{\gamma}{J(k_1,k_2,k_3,...)} .\end{equation}

The results of this sections are adapted from the formalism used on section \ref{zero} where
more details are presented about the interplay of the symbolic string  $(k_0,k_1,k_2,k_3,...)$ and the
limit value $J$.

In order to estimate the value of the fraction expansion of  $J(k_0,k_1,k_2,...)$ we will have to consider the two functions
$$ f_0(x) = \alpha + \frac{\gamma}{x}$$
and
$$ f_1(x) = 1- \alpha - \frac{\gamma}{x}.$$

It follows from Lemma \ref{tri} that the  values of the Jacobian $J(k_0,k_1,...,k_r,..)$ are the possible limits of iterations of the form
\begin{equation} \label{dada} f^{n_l}_{u_l} \circ f^{m_l}_{v_l} \circ f^{n_{l-1}}_{u_{l-1}} \circ f^{m_{l-1}}_{v_{l-1}} \circ ... \circ f_{u_1}^{n_1} \circ f_{v_1}^{m_1}\, (1/2),
\end{equation}
where $u_j,v_j\in\{1,2\}$, $j=1,...,l$,  and $l \to \infty.$ The values $m_j$ and $n_j$ will depend of the successive  changes (or, not) from $k_h$ to $k_{h+1}$, $h \in \mathbb{N},$ on the
string $ (k_0,k_1,k_2,...)$ (according to (\ref{li1}) and (\ref{li2})).

\medskip

Note that  $f_1$ is associated to changing symbols on the string and $f_0$ to not changing symbols. For example,
for $k=(1,2,2,1,2,2,1,....)$ we get
$$J^6 (k) =      (\,    f_1 \circ f_0 \circ f_1 \circ  f_1    \circ f_0\,)\, (1/2)$$
and
$$J^7 (k) =          (\,f_1 \circ f_0 \circ f_1 \circ  f_1    \circ f_0 \circ f_1 )\,(1/2).$$

We denote by $r= \frac{\alpha + \sqrt{\alpha^2 + 4 \gamma  }}{2} $ the positive fixed point of $f_0$. This fixed point is contracting. The function $f_1$ has two positive fixed points: $\tilde{R}= \frac{    1-\alpha - \sqrt{(\alpha-1)^2 - 4 \gamma  }}{2}$ (which is expansive) and $R=\frac{    1-\alpha + \sqrt{(\alpha-1)^2 - 4 \gamma  }}{2}$ (which is contractive).

We claim that
$$ 0<\tilde R<1-R<r<1/2<R<1-\alpha.$$
Indeed, $\tilde R+R=1-\alpha$, where $R<1-\alpha$ and $\tilde R=1-\alpha-R<1-R$. We have $x^2-\alpha x-\gamma<0$ for $0\le x<r$ and $x^2-\alpha x-\gamma>0$ for $x>r$. Since $(1/2)^2-\alpha\cdot (1/2)-\gamma=(1-2\alpha-4\gamma)/4>0$, it follows that $1/2>r$. On the other hand, $x^2-(1-\alpha)x+\gamma<0 \iff \tilde R<x<R$, and, since $(1/2)^2-(1-\alpha)\cdot (1/2)+\gamma=(-1+2\alpha+4\gamma)/4<0$, we have $1/2<R$. Finally, $r=f_0(r)>f_0(R)=\alpha+\dfrac{\gamma}{R}=\alpha+(1-\alpha-R)=1-R$ (since $f_0$ inverts orientation in $(0,+\infty)$.

\begin{lemma} $f_0$ and $f_1$ are contractions on the interval $I=[1-R,R]$ for some metric $\tilde{d}$.

\end{lemma}

{\bf Proof:}
We claim that the interval  $I=[1-R,R]$ (note that $1/2$ belongs to $I$) is such that $f_j(I)\subset I$, for $j=0, 1$.

\medskip

First we consider $f_0$.
Note that $f_0(R)=\alpha+\dfrac{\gamma}{R}=\alpha+(1-\alpha-R)=1-R$.
We claim that $f_0(1- R)  < R.$
Indeed, this means $ \alpha - \alpha R + \gamma <R- R^2$, which is equivalent to $R^2 + \alpha - \alpha R + \gamma-R<0$.
As $R^2=R- \alpha R - \gamma $ (because $R$ is a fixed point of $f_1$), the above one is equivalent to $\alpha - 2\alpha R <0$ which is satisfied because $\alpha >0$ and $1/2<R$.

Now we study $f_1$. Note that $f_1(I)=[f_1(1-R),f_1(R)]=[f_1(1-R),R]$. As $f_1$ is concave, and $(\tilde R, \tilde R), (R, R)$ are points of the graph,  we have $f_1(x)>x$, for $\tilde R<x<R$. From this we get that $f_1(1-R)>1-R$.

 We will have to consider an IFS of the form (\ref{dada}). Then, it will be necessary that $f_0 ( f_1 (1-R))> 1- R$.
This, follows from the fact that $f_0$ is monotonous decreasing and $f_1(1-R)<R.$
From the iteration dynamics (\ref{dada}) we get that the possible values of $J$ will be on the interval
$[1-R,R].$

Is not true that the modulus of the derivatives  $|f_0 '|$ and $|f_1 '|$ are always smaller than $1$ on the interval $[1-R,R]$, but we claim that $f_0$ and $f_1$ are contractions for a  distance obtained from a certain differentiable Riemannian metric on the interval $I$.

The above claim implies that there exists a  natural number
$n$, such that, compositions of $n$ times (using the  functions $f_0$ and $f_1$ in any way) are strong contractions (on the usual metric) and the sequence
$(x_n)$, $n \in \mathbb{N},$  given by $x_n=\frac{\mu(k_0,k_1,k_2,...,k_n)}{\mu(k_1,k_2,...,k_n)}$ exponentially converges (to $J(k_0,k_1,k_2,\dots)$), for any choice of $k_0, k_1, k_2,\dots$. Moreover,  $J$
will be a H\"older function of $k_0,k_1,k_2,\dots$.

In order to show the claim, note that
for small $\epsilon>0$ the interval $J=(1-R-\epsilon,R+\epsilon)$ is such that the interval $f_j(J)$ is strictly contained in $J$, for $j=0, 1$.

Indeed, $|f_0'(x)|=\gamma/x^2<1$, for $x\ge r$ (because $\sqrt{\gamma}<r$, which follows from $(\sqrt{\gamma})^2-\alpha\cdot \sqrt{\gamma}-\gamma=-\alpha\cdot \sqrt{\gamma}<0$). Therefore,  $1-R-f_0(R+\epsilon)=|f_0(R+\epsilon)-f_0(R)|<\epsilon$, that is, $f_0(R+\epsilon)>1-R-\epsilon$. Moreover, as $f_0(1-R)<R$, if $\epsilon$ is small enough we get $f_0(1-R-\epsilon)<R<R+\epsilon$.

On the other hand, as $f_1(x)>x$, for $\tilde R<x<R$ and $1-R>\tilde R$, we get for $\epsilon$ small enough $1-R-\epsilon>\tilde R$, and then, $f_1(1-R-\epsilon)>1-R-\epsilon$. Finally note that  $f_1(x)<x$, for $x>R$, and from this follows $f_1(R+\epsilon)<R+\epsilon$.

The proof of the claim that $f_0$ and $f_1$ are contractions on some metric follows from the fact that given and interval $J$, there exists a differentiable metric in $J$ which is contracted by any M\"obius transformation (of the form $M(x)=\frac{ax+b}{cx+d}$) which maps $J$ in an interval strictly contained in  $J$.

Considering a  conjugation which takes $J$ exactly on the interval $(-1,1)$, we get that, it is enough to prove this result for the interval $(-1,1).$

We will show that  any M\"obius transformation $M(x)=\frac{ax+b}{cx+d}$ which takes $(-1,1)$ strictly inside $(-1,1)$ strictly  contracts the metric $\rho(x)dx$, where $\rho(x)=\frac1{1-x^2}$.
This means $|M'(x)\rho(M(x))|<|\rho(x)|$, for all, $x\in(0,1)$ (we note that this metric $\rho(x)dx$ is the restriction of the Poincar\'e hyperbolic metric on the unit disk to the interval $(-1,1)$; we refer the reader to sections 3.3 and 3.4 in \cite{Bea} for  general results on the action of  M\"obius transformation on the hyperbolic metric on the disk).

In order to show that we point out that for any
$a\in (-1,1)$, the M\"obius transformation ${\cal M}_a(x)=\frac{x-a}{1-ax}$ is a diffeomorphism from $(-1,1)$ to $(-1,1)$ which maps $a$ to $0$ and $0$ to $-a$. Moreover, is a preserving orientation isometry  for this metric.
Indeed,
$${\cal M}_a'(x)=\frac{1-a^2}{(1-ax)^2}=\frac{1-{\cal M}_a(x)^2}{1-x^2}=\frac{\rho(x)}{\rho({\cal M}_a(x))}.$$

If $M(x)=\frac{ax+b}{cx+d}$ takes $(-1,1)$ in an interval strictly contained in  $(-1,1)$,
the same  happens for the M\"obius transformation
$$\tilde M(x)={\cal M}_{M(x)}\circ M\circ{\cal M}_{-x}$$ (where $x \in (-1,1)$).

We have
$\tilde M(0)=0$ and
$$\tilde M'(0)={\cal M}'_{M(x)}(M(x))\cdot M'(x)\cdot{\cal M}'_{-x}(0)=$$
$$\frac{\rho(M(x))}{\rho(0)}\cdot M'(x)\cdot\frac{\rho(0)}{\rho(x)}=\frac{\rho(M(x))}{\rho(x)}M'(x),$$
therefore, all we have to show is  $|\tilde M'(0)|<1$.

Note that  $\tilde M$ is a M\"obius transformation that  maps $(-1,1)$ in an interval strictly contained  in $(-1,1)$ and satisfies $\tilde M(0)=0$. From this we get
necessarily $\tilde M(x)=\frac{x}{sx+t}$, where $|t+s|\ge 1$ e $|t-s|\ge 1$, and, moreover, at least one of the inequalities is strict. This implies that $|t|>1$ and $|s|\le |t|-1$. From this we get that $|M'(0)|=|1/t|<1$ and this shows the main claim.

In this way there exists $0<\lambda<1$ such that  $f_0$ and $f_1$ contract distances by a factor $\lambda$  for some distance $\tilde{d}(x,y)$ (induced by the Riemannian metric) on the metric space $[R-1,R]$.

\qed

\medskip

\begin{corollary} The Jacobian $J: \{1,2\}^\mathbb{N} \to (0,1)$ is a H\"older continuous function.

\end{corollary}

{\bf Proof:}
Denote by $0<\lambda<1$ a contraction constant (on the hyperbolic distance $\tilde{d}$) for both $f_0$ and $f_1$. We denote by $d$ the usual distance on  $\Omega$.

Given a point $k\in \Omega$, $k= (k_0,k_1,k_2,...,k_n,...)$, the Jacobian $J(k)=J(k_0,k_1,k_2,...,k_n,...)$ is obtained via the limit of expression (\ref{dada}). Given also another point $q=(q_0,q_1,q_2,...,q_n,...) \in \Omega$, if $d(q,k)=2^{-t}$, $t \in \mathbb{N}$, then,
$(k_0,k_1,k_2,...,k_n,...)$ and $(q_0,q_1,q_2,...,q_n,...)$ coincide until order $t$.  We want to compare $J^s(k) $ and $J^s(q)$ for $s$ much more larger than $t$.
%
% We will get, respectively, two different expressions.%
% For $k$ and $s>t$ we get
% \begin{equation} \label{era1}  f^{n_r}_0 \circ f^{m_r}_1 \circ f^{n_{r-1}}_0 \circ f^{m_{r-1}}_1 \circ ... \circ f_0^{n_1} \circ f_1^{m_1}\, \circ f_{v_r}^{l_r} \circ f_{u_r}^{y_r}\circ ...\circ f_{v_1}^{l_1} \circ f_{u_1}^{y_1} (1/2),
%\end{equation}
%where  $u_j,v_j\in\{1,2\}$, $l_j,y_j \in \mathbb{N},$ for $j=1,...,r.$
%
% For $q$ and $s>t$ we get
% \begin{equation} \label{era2}  f^{n_r}_0 \circ f^{m_r}_1 \circ f^{n_{r-1}}_0 \circ f^{m_{r-1}}_1 \circ ... \circ f_0^{n_1} \circ f_1^{m_1}\, \circ f_{\tilde{v}_s}^{\tilde{l}_s} \circ f_{\tilde{u}_s}^{\tilde{y}_s}\circ ...\circ f_{\tilde{v}_1}^{\tilde{l}_1} \circ f_{\tilde{v}_1}^{\tilde{y}_1} (1/2),
%\end{equation}
%where  $\tilde{u}_j,\tilde{v}_j\in\{1,2\}$, $\tilde{l}_j,\tilde{y}_j \in \mathbb{N},$ for $j=1,...,s.$
%
For $s>t$ fixed, we get from $(\ref{dada})$,
\begin{equation} \label{era3} J^s (k)= f^{n_r}_0 \circ f^{m_r}_1 \circ f^{n_{r-1}}_0 \circ f^{m_{r-1}}_1 \circ ... \circ f_0^{n_1} \circ f_1^{m_1}\, (d_s),
\end{equation}
and
 \begin{equation} \label{era}  J^s (q) = f^{n_r}_0 \circ f^{m_r}_1 \circ f^{n_{r-1}}_0 \circ f^{m_{r-1}}_1 \circ ... \circ f_0^{n_1} \circ f_1^{m_1}\, (c_s),
\end{equation}
where $n_r+m_r+...+n_1+m_1=t$ and $d_s, c_s \in [1-R,R]$.

As $f_0$ and $f_1$ are $\lambda$-contractions we get
$$  \tilde d(J^s (k),J^s (q)) < \lambda^t \,\tilde{d}(d_s,c_s) \leq \lambda^t \, \, D,$$
where we denote by $D$ the diameter of $[1-R,R]$ according to $\tilde{d}.$
Since the distance $\tilde d$ and the usual (Euclidean) distance are equivalent on $[1-R,R]$, there is a constant $K>0$ such that $|J^s (k) - J^s (q)|\leq K\cdot\lambda^t$ for every $s>t$. It follows that (as $s\to+\infty$) $|J (k) - J (q)|\leq K\cdot\lambda^t$.

If $\lambda\leq 1/2$ we get  $  |J (k) - J (q)| \leq (1/2)^t \, \, K =d(k,q)\, K$, and therefore $J$ is a Lipchitz function.
If $\lambda> 1/2$ consider $\delta>0$ such that $\lambda= (1/2)^\delta.$ In this case we get $  |J(k) - J (q)| \leq \lambda^t \, \, K = (\, (1/2)^\delta\,)^t \, \, K =  (\, (1/2)^t\,)^\delta \, K \,=\, d(k,q)^\delta\, K$ and therefore $J$ is a
 $\delta$-H\"older function.

\qed

\medskip

Remark:  The image of $J(k_0,k_1,k_2,\dots)$ is the  attractor for the Iterated Function System \cite{IFS} defined by $f_0$ and $f_1$ acting on the interval $I=[1-R,R]$. We will study when the image of $J$ is a Cantor set or an interval.

\bigskip

\begin{figure}
	\begin{center}
		\includegraphics[width=8.5cm,height=6cm]{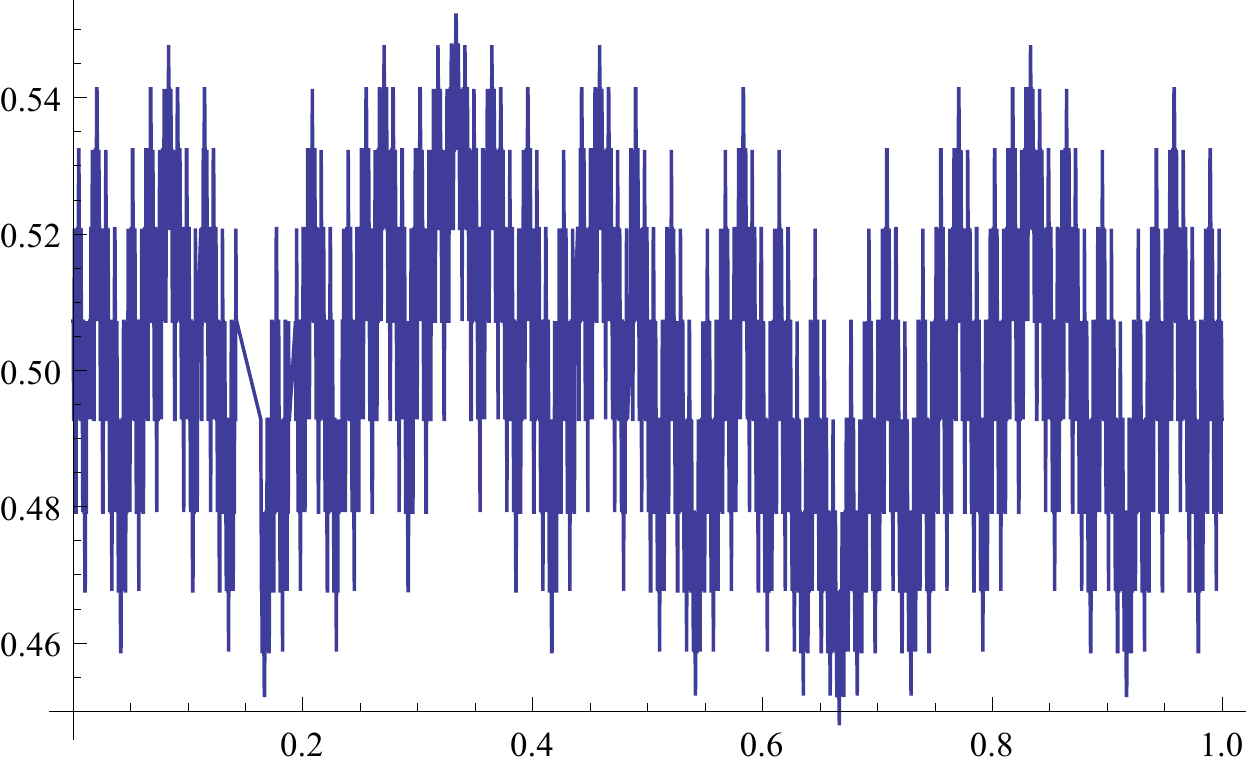}
		\caption{ The graph of the values of the Jacobian $J^{12}$  when $\cos (2 \, \theta) =0.12$ and $\beta=5.6$. The transition value of $\beta$ is equal to $2.341$. Above the points on the interval are associated with points in $\{1,2\}^\mathbb{N}$ using the binary expansion with symbols $1$ and $2$ (by this we mean: on the binary expansion we associate $0$ to $1$, and, $1$ to $2$). We considered strings with $12$ symbols $k_0,k_1,...,k_{11}\in\,\{1,2\} $.
		}
		\end{center}
\end{figure}

\begin{figure}
	\begin{center}
		\includegraphics[width=8.5cm,height=6cm]{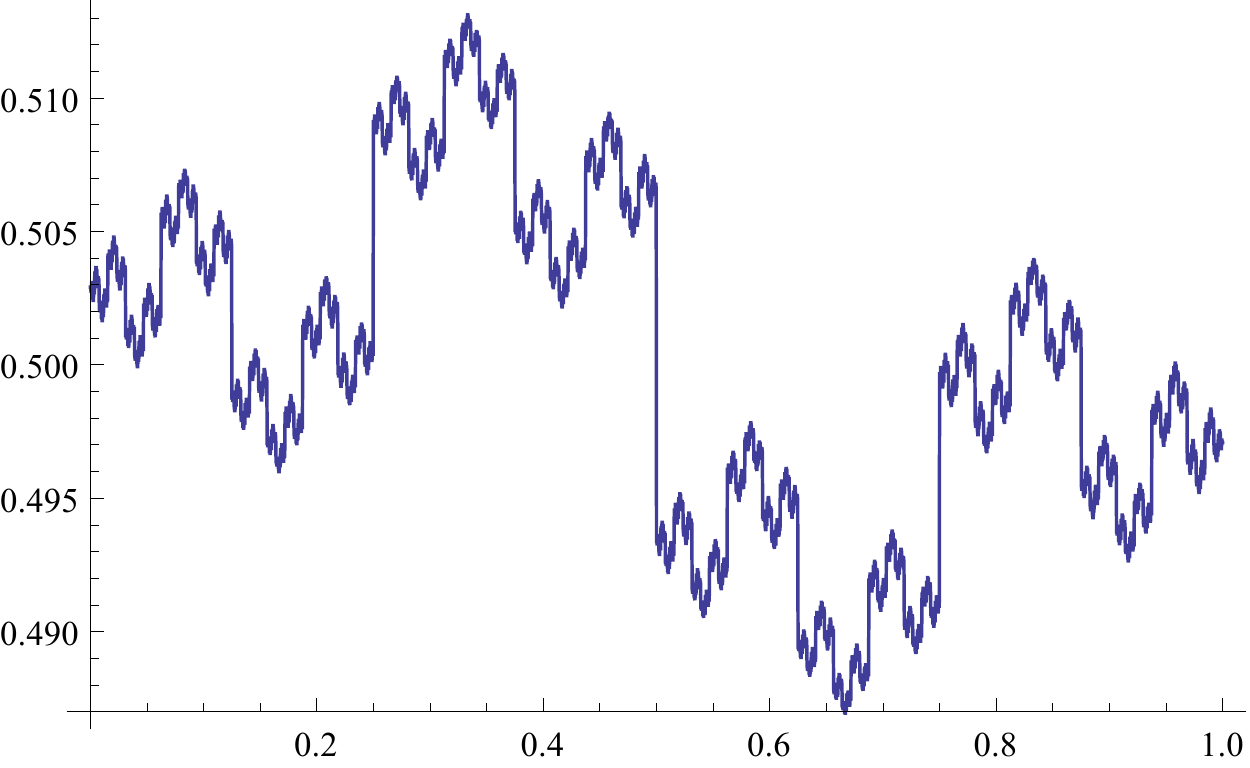}
		\caption{ The graph of the values of the Jacobian $J^{12}$  when $\cos (2 \, \theta) =0.12$ and $\beta= 0.8$. The transition value of $\beta$ is equal to $2.341$.  Above the points on the interval are associated with points in $\{1,2\}^\mathbb{N}$ using the binary expansion with symbols $1$ and $2$ (by this we mean: on the binary expansion we associate $0$ to $1$, and, $1$ to $2$). We considered strings with $12$ symbols $k_0,k_1,...,k_{11}\in\,\{1,2\} $.
		}
	\end{center}
\end{figure}

\begin{proposition} \label{tr} The image of $J$ is a Cantor set, if and only if, $0<\gamma<\alpha(1-2\alpha)$. In the other case  it is an interval. The transition value for $\beta$ where  $\gamma=\alpha(1-2\alpha)$  corresponds to the value
$$\beta = \frac{1}{2} \log (\frac{4}{ \cos^2 (2\,\theta)} - 1).$$

\end{proposition}

\begin{proof}

 As $f_0$ and $f_1$ are contractions on $I=[1-R,R]$ (in a suitable metric) the image of $J$ is a regular Cantor set or an interval.  We can say exactly when such alternative occurs. If $f_0(I)$ and $f_1(I)$ are disjoint the attractor will be a Cantor set with Hausdorff dimension strictly between $0$ and $1$. If $f_0(I)\cup f_1(I)=I$ the attractor is all the interval $I$.

Observe that $f_0(I) = [1-R, f_0(1-R)]$ and $f_1(I)=[f_1(1-R),R]$. In this way,  $f_0(I)$ and $f_1(I)$ are disjoint, if and only if,  $f_0(1-R)<f_1(1-R)$. As $f_0(1-R)=\alpha+\frac{\gamma}{1-R}$ and $f_1(1-R)=1-\alpha-\frac{\gamma}{1-R}$ we have
  $$f_0(1-R)<f_1(1-R)\iff \alpha+\frac{\gamma}{1-R}<1-\alpha-\frac{\gamma}{1-R} \iff$$
  $$ \frac{2\gamma}{1-R}<1-2\alpha\iff 1-R>\frac{2\gamma}{1-2\alpha} \iff R<1-\frac{2\gamma}{1-2\alpha}. $$

The solutions of $x^2 -(1-\alpha)x +\gamma =0$ are $\tilde R$ and $R$. As   $\frac{2\gamma}{1-2\alpha}<\frac{2\gamma}{4\gamma}=1/2$ we get $1-\frac{2\gamma}{1-2\alpha}>1/2>\tilde R$. It follows that
  $$R<1-\frac{2\gamma}{1-2\alpha}\iff (1-\frac{2\gamma}{1-2\alpha})^2-(1-\alpha)(1-\frac{2\gamma}{1-2\alpha})+\gamma>0.$$
  Therefore
  $$f_0(1-R)<f_1(1-R) \iff (1-\frac{2\gamma}{1-2\alpha})^2-(1-\alpha)(1-\frac{2\gamma}{1-2\alpha})+\gamma>0$$
$$\iff -4\gamma(1-2\alpha)+4\gamma^2+\alpha(1-2\alpha)^2+2\gamma(1-2\alpha)-2\alpha\gamma(1-2\alpha)+\gamma(1-2\alpha)^2>0.$$

The above condition can be expressed as
$$0<4\gamma^2-(1-2\alpha)(1+4\alpha)\gamma+\alpha(1-2\alpha)^2=(4\gamma+2\alpha-1)(\gamma-\alpha(1-2\alpha)).$$
As $4\gamma+2\alpha-1<0$, this condition is equivalent to  $\gamma-\alpha(1-2\alpha)<0$.

Therefore, the conclusion is that the attractor is a regular Cantor set, if and only if,
$$0<\gamma<\alpha(1-2\alpha).$$
%\textbf{(note that if $0<\alpha<1/4$, this condition  defines a proper subinterval of $(0, \frac{1-2\alpha}4)$ for the possible values of $\lambda$, but, if $1/4\le \alpha<1/2$, the condition is automatically satisfied).}

%\textbf{For the other possibility, that is, the case $\alpha(1-2\alpha)\le \lambda<\frac{1-2\alpha}4$ (which is only possible if  $0<\alpha<1/4$), we get $f_0(I)\cup f_1(I)=I$, and the attractor will be the whole interval $I$.}

From definition of $\alpha$ and $\gamma$ we have $\gamma =\frac{1}{4}\, (1-  2\alpha) \cos^2 (2 \, \theta)$.
Then the inequality $0<\gamma<\alpha(1-2\alpha)$ is equivalent to $ \frac{\cos^2 (2 \, \theta)}{4} <\alpha$.
As $ \alpha= \frac{ \Phi_\beta}{2}= \frac{1}{e^{2\, \beta}+1}$ we finally get
$$ \frac{\cos^2 (2 \, \theta)}{4} <\frac{1}{e^{2\, \beta}+1} \iff  \beta < \frac{1}{2}  \log (\frac{4}{ \cos^2 (2\,\theta)} - 1).    $$
The final conclusion is that the image of $J$ is a regular Cantor set, if and only if,
$$\beta < \frac{1}{2}  \log (\frac{4}{ \cos^2 (2\,\theta)} - 1).$$
When $\beta $ is larger we get  that the image of $J$ is an interval.

\end{proof}

\medskip

As the probability $\mu_\beta$ is mixing (because the potential is H\"older) in particular we have that for any cylinders $A=[a_1,...,a_k]$ and $B=[b_1,...,b_l]$:
\[\lim_{n\to\infty}\mu(A\cap\sigma^{-n}B)=\mu(A)\mu(B).    \]

This means
			\begin{equation}\label{mixing}
			\lim_{n\to\infty}\sum_{j_1,...,j_n}\mub(a_1,...,a_k,j_1,...,j_n,b_1,...,b_l)=\mub(a_1,...,a_k)\mub(b_1,...,b_l).
			\end{equation}

One can show the following precise result concerning the speed of convergence (see \cite{Br} for the computations):

\begin{align*}
			\lim_{n\to\infty}&\sum_{j_1,...,j_n}\mub(a_1,a_2,a_3,...,a_k,j_1,...,j_n,b_1,...,b_l)\\
			&= \bigg[ \frac{1}{2}\bigg(1+\frac{\beta_{a_1}}{\beta_{a_2}}(-1+\phib)\bigg)\mub(a_2,...,a_k)\,+\\
			&\hspc \sum_{j_1,...,j_n}\frac{\beta_{a_1}}{2}(-1+\phib) \bigg(\frac{-1}{2\beta_{a_{2}}}+\frac{\beta_{a_{2}}}{2}\bigg)\mub(a_3,...,a_k) \bigg]\mub(b_1,...,b_l)\\
			&= \mub(a_1,a_2,...,a_k)\mub(b_1,...,b_l).
			\end{align*}

\section{LDP} \label{LDP}

In this section we will present some results which are similar but more complex that the ones in section 5 in \cite{LMMM}.

\medskip

In this section we will prove Theorem \ref{teo2}.

\medskip

		\begin{definition}
			For $x=(x_1,x_2,\cdots)\in\Omega=\{1,2\}^{\N}$ and $A:\Omega\to\R$, we define the $n$-Birkhoff sum for $A$ and $x$ as
			$$ S_n(A,x) := \sum\limits_{j=0}^{n-1} A(\sigma^j(x)) = A(x)+A(\sigma(x))+\cdots+A(\sigma^{n-1}(x)). $$
		\end{definition}

		\begin{definition}
			Suppose $A:\Omega\to\R$ is a H\" older function. We define for each  $t\in\R,$
			
			$$Q_n(t) := \int e^{t\, S_n(A,x) }d \mub (x).$$
			
			In the case $A$ depends just on the first coordinate we get	
			$$Q_n(t) =  \sum_{x_0}\, \sum_{x_1}\, ...\sum_{x_n} e^{t\, (  A(x_0) + A(x_1) +...+ A(x_n) ) }\, \mub ( x_0,x_1,...x_n).$$
		\end{definition}

		 The free energy on time $t$ is	
		$$ c (t) := \lim_{n \to \infty}  \frac{1}{n}\,\, \log \int e^{t\,S_n(A,z)} d \mub (z) = \lim_{n \to \infty}  \frac{1}{n}\,\, \log Q_n(t),$$
		
		\noindent for each $t \in \mathbb{R}$. 				
		
		 Remember that
		
		$$\displaystyle\delta(t)= \sum_{j}  e^{t\, \, A(j)}\,\,\,\text{ and }\,\,\,\alpha(t)=  \sum_{j} \beta_{j}\, e^{t\, \, A(j)}.$$

		%	Note que a expressão explicita de $c(t)$ permite estimar (via transformada de Legendre) a função de desvio $I$.
		%	\medskip

		\begin{remark}\label{alphamenordelta}
			Note that $|\alpha(t)|< | \delta(t)|$, because $|\beta_j|\le 1$ and $\beta_1 = -\beta_2$. Moreover, note that $\delta(t) > 0$.
		\end{remark}

		%ANTIGO: \frac{1}{2}\log \bigg(\Big[ e^{t\, A(1)}+e^{t\,A(2)}\Big]^{2}-\Big[ \beta_{1}e^{t\, A(1)}+\beta_2 e^{t\, A(2)}\Big]^{2}\bigg)-\log(2).
		We will show that for any $t\in \mathbb{R}$,
			\begin{align*}
			&c(t)=\log\left( \frac{ \phib\delta(t) + \sqrt{\phib^2\delta(t)^2 + 4(-1+\phib)(\alpha(t)^2-\delta(t)^2)}}{4}\right).
			\end{align*}

	The function $c(t)$ is differentiable on $t$.

		\medskip
		\begin{example}
			We will estimate $Q_3(t)$. From Theorem \ref{medidacilindro} we get
			
			\begin{align*}
				\mub(j_0,j_1,j_2,j_3) = &\frac{\mub(j_1,j_2,j_3)}{2} + \frac{(-1+\phib)\beta_{j_0} \beta_{j_1}}{4}\mub(j_2,j_3)\,+ \\
					&\frac{(-1+\phib)^2\beta_{j_0} \beta_{j_2}}{16} +  \frac{(-1+\phib)^3\beta_{j_0} \beta_{j_3}}{16},	
			\end{align*}

			\noindent therefore,			
			\begin{align*}
				Q_3(t) &= \sum_{j_0}\, \sum_{j_1}\, \sum_{j_2}  \sum_{j_3} e^{t\, (  A(j_0) + A(j_1) +A(j_2) + A(j_3)) }\, \mub ( j_0,j_1,j_2,j_3)\\
				&= \bigg[\frac{1}{2}\sum_{j_0}\, e^{t\,   A(j_0)  }\,\bigg]\,\sum_{j_1}\, \sum_{j_2} \sum_{j_3}\, e^{t\, (   A(j_1) +A(j_2)+ A(j_3)) }\, \mub (j_1,j_2,j_3)\\
				&+ \frac{(-1+\phib)}{4}\bigg[\sum_{j_0}e^{t\, A(j_0)}\beta_{j_0}\bigg]\, \bigg[\sum_{j_1}e^{t\,A(j_1)}\beta_{j_1}\bigg]\, \sum_{j_2,j_3}   e^{t\, ( A(j_2)+ A(j_3)) } \mub(j_2,j_3) \\
				&+\frac{(-1+\phib)^2}{16}\sum_{j_0}\, \sum_{j_1}\, \sum_{j_2}   \sum_{j_3} e^{t\, (  A(j_0) + A(j_1) +A(j_2))+ A(j_3) }\beta_{j_0} \beta_{j_{2}}\\
				&+\frac{(-1+\phib)^3}{16}\sum_{j_0}\, \sum_{j_1}\, \sum_{j_2}   \sum_{j_3} e^{t\, (  A(j_0) + A(j_1) +A(j_2))+ A(j_3)) }\beta_{j_0} \beta_{j_{3}}\\
				&=\frac{1}{2}\delta(t)Q_2(t) +  \frac{(-1+\phib)}{4}\alpha(t)^2 Q_1(t) +\frac{\phib(-1+\phib)^2}{16}\alpha^2(t)\delta^2(t).
			\end{align*}
		\end{example}
		\medskip

		 In the general case we get:
		
		\begin{theorem}\label{teoQ} For all $n\in \mathbb{N}$ and $t\in\mathbb{R}$
			
			\begin{equation}\label{main1}
			Q_n (t)=\left[\begin{array}{l}\frac{1}{2}\delta(t)\,Q_{n-1}(t) + \frac{(-1+\phib)}{4}\alpha(t)^2\,Q_{n-2}(t)+ \frac{(-1+\phib)^2}{8}\delta(t)\,\alpha(t)^2\, Q_{n-3}(t) \\ \\
			\,\,+  \frac{(-1+\phib)^3}{16}\delta(t)^2\,\alpha(t)^2\, \, Q_{n-4}(t)+\frac{(-1+\phib)^4}{32}\delta(t)^3\,\alpha(t)^2\, \, Q_{n-5}(t)+...+\\  \\      +\frac{(-1+\phib)^{n-3}}{2^{n-2}}\, \delta(t)^{n-4}\,\alpha(t)^2\, \, Q_{ 2}(t)
			+ \frac{(-1+\phib)^{n-2}}{2^{n-1}}\, \delta(t)^{n-3}\,\alpha(t)^2\, \, Q_{ 1}(t)\\ \\
			+ \frac{\phib(-1+\phib)^{n-1}}{2^{n+1}}\alpha^2(t)\delta^{n-2}(t).\end{array}\right] .
			\end{equation}
			
		\end{theorem}
		
		For the proof  see Appendix (Theorem \ref{teoQ1})

\medskip
				
		\begin{proposition}\label{recurrence_relation}
			\[Q_{n+2}(t) = (-1+\phib)\frac{\alpha(t)^2-\delta(t)^2}{4}\,Q_{n}(t) + \frac{\phib}{2}\delta(t)\,Q_{n+1}(t).\]
		\end{proposition}

\medskip
For the proof see Appendix (Theorem \ref{recurrence_relation1})

\medskip

The equation presented in  Proposition \ref{recurrence_relation} (positive temperature)  is more complex when compared with the analogous  one (Proposition 5.2) in \cite{LMMM} (zero temperature). Note that at zero temperature $\beta\to\infty$ and  $\phib\to 0$. At zero temperature the last term above disappears. In the present case $\phib > 0$, and we need a recurrence relation.

Note that for fixed $t$ and $\beta$ the expressions $(-1+\phib)\frac{\alpha(t)^2-\delta(t)^2}{4}$ and $\frac{\phib}{2}\delta(t)$ do not depend on $n$. For each  $t$ we get a second order recurrence relation which will solved by using a result we get from   \cite{Scheinerman}:
		
		\bigskip
		
		\begin{theorem}\label{scheinerman}
			Given the real numbers $s_1$ and $s_2$ suppose that $r_1, r_2$ are the roots of the equation $x^2-s_1\,x - s_2 = 0$. If $r_1 \ne r_2$, then, any solution of the recurrence equation
			$$a_n = s_1 a_{n-1} + s_2 a_{n-2}$$
			is of the form
			$$a_n = c_1 r_1^n + c_2 r_2^n,$$
			where $c_1, c_2$ are  constants.
		\end{theorem}
	
\noindent		
\textbf{Proof of Theorem \ref{teo2}:}		 Let's check that the recurrence relation presented in Proposition \ref{recurrence_relation} satisfies the  hypothesis of theorem \ref{scheinerman}. The roots of
$$x^2 - \frac{\phib}{2}\,\delta(t)\,x - (-1+\phib)\frac{\alpha(t)^2\,-\,\delta(t)^2}{4}$$
are
		
		\begin{align*}
			r_1 = r_1(t) =& \frac{\frac{\phib}{2}\delta(t) + \sqrt{\frac{\phib^2}{4}\delta(t)^2 + 4(-1+\phib)\frac{\alpha(t)^2-\delta(t)^2}{4}}}{2}\\
			&= \frac{ \phib\delta(t) + \sqrt{\phib^2\delta(t)^2 + 4(-1+\phib)(\alpha(t)^2-\delta(t)^2)}}{4}
		\end{align*}
		
		\noindent and
		\begin{align*}
			r_2 = r_2(t) =& \frac{\frac{\phib}{2}\delta(t) - \sqrt{\frac{\phib^2}{4}\delta(t)^2 + 4(-1+\phib)\frac{\alpha(t)^2-\delta(t)^2}{4}}}{2}\\
			&= \frac{ \phib\delta(t) - \sqrt{\phib^2\delta(t)^2 + 4(-1+\phib)(\alpha(t)^2-\delta(t)^2)}}{4}.
		\end{align*}
		
		 By Remark \ref{alphamenordelta} it follows that $\phib\delta(t) > 0$ and moreover $4\,(-1+\phib)\,(\alpha(t)^2-\delta(t)^2) > 0$, because $\phib \in (0,1)$. In this case, $r_1 \ne r_2$. Moreover, $r_1 > -r_2 > 0$. Indeed,
		
		$$  \sqrt{\phib^2\delta(t)^2 + 4(-1+\phib)(\alpha(t)^2-\delta(t)^2)} > \sqrt{\phib^2\delta(t)^2} = \phib\delta(t) > 0.$$
		
		 Therefore,
				$$ Q_n(t) = c_1(t) r_1(t)^n + c_2(t) r_2(t)^n$$
with $r_1 > -r_2 > 0$. 	It follows that, for any $t$,
		\begin{align*}
			c(t) &= \lim_{n\to \infty}\frac{1}{n}\log Q_n(t) = \log r_1(t)\\
			&= \log\left( \frac{ \phib\delta(t) + \sqrt{\phib^2\delta(t)^2 + 4(-1+\phib)(\alpha(t)^2-\delta(t)^2)}}{4}\right).
		\end{align*}

		\qed
		
\medskip

	\section{Some results on zero temperature} \label{zero}

	In this section we will complement some results obtained in \cite{LMMM}. We denote by $\mu$ the zero temperature quantum spin probability on $\{1,2\}^\mathbb{N}$ as described in \cite{LMMM} which  is ergodic but not mixing. We want to prove Theorem \ref{zeze}. Initially we remember some definitions and results from \cite{LMMM}.
	
	\medskip
	
	We denote in this section $  \beta_1= \sin(2\, \theta)$, $\beta_2=-\beta_1$, $\gamma=\frac{1}{4} \bigg(1-\beta_1^2\bigg) = \mu(1,1),$
			$$a(k_0,k_1)= \frac{1}{2} \bigg(1-\frac{\beta_{k_0}}{\beta_{k_1}}\bigg)$$
	and
	$$ b(k_0,k_1)= \frac{1}{4} \beta_{k_0}\bigg(\frac{1}{\beta_{k_{1}}}-\beta_{k_{1}}\bigg).$$
	At zero temperature we  get

\noindent 	
	a) if $k_0=k_1$, then $a(k_0,k_1)=0$ and $0< b(k_0,k_1)=\gamma=\frac{1}{4} (1-\beta_1^2)<\frac{1}{4}$
	
\noindent	
	b) if $k_0\neq k_1$, then $a(k_0,k_1)=1$ and $-\frac{1}{4}<b(k_0,k_1)=-\gamma=\frac{1}{4} (\beta_1^2-1)<0$
	
	\medskip
	
	Using $a$ and $b$ the measure $\mu$ can be computed recursively in finite cylinders\footnote{comparing this recursively relations with those for positive temperature we get that $\mu(C)=\lim_{\beta\to\infty} \mu_\beta (C)$ for any cylinder set $C\subset \Omega$ and therefore $\mu$ is the weak* limit of $\mu_\beta$ as $\beta \to \infty$.}  by $\mu(\{1,2\}^\mathbb{N} )=1$, $\mu(1)=\mu(2)=1/2$ and, for $n\ge 1$,
	$$\mu(k_0,k_1,k_2,...,k_n)=a(k_0,k_1)\mu(k_1,k_2,k_3,...,k_n)+b(k_0,k_1)\mu(k_2,k_3,...,k_n),$$
	which can be rewritten as
	\begin{equation} \label{eq1}
\frac{\mu(k_0,k_1,k_2,...,k_n)}{\mu(k_1,k_2,...,k_n)}=a(k_0,k_1)+b(k_0,k_1)\frac{1}{\frac{\mu(k_1,k_2,k_3,...,k_n)}{\mu(k_2,k_3,...,k_n)}}.
	\end{equation}
	
	The Jacobian $J$ of the invariant probability $\mu$ is given by
	$$J(k_0,k_1,k_2,...,k_n,...)=\lim_{n\to \infty}\frac{\mu(k_0,k_1,k_2,...,k_n)}{\mu(k_1,k_2,...,k_n)}$$
	which exists $\mu$ almost everywhere and satisfies  the lemma below (see \cite{LMMM}).

	\begin{lemma}\label{jexpansion}
	
	\begin{equation} \label{cfg1} J(k_0,k_1,k_2,...) = a(k_0,k_1) + b(k_0,k_1)\frac{1}{J(k_1,k_2,k_3,...)}
	\end{equation}
and \footnotesize
		\begin{align*}
		J(k_0&,k_1,k_2,...) = \lim_{n} \left[a(k_0,k_1) +  b(k_0,k_1)\,\frac{1}{a(k_1,k_2) +  b(k_1,k_2)\frac{1}{...a(k_{n-1},k_n)+b(k_{n-1},k_n)\frac{1}{1/2}}}\right],
		\end{align*} \normalsize
		if the limit exists.
	\end{lemma} In this sense ${J}$ has an expression in continued fraction (according to \cite{LMMM}).
The following result is mentioned in \cite{LMMM}. Below we will provide a complete proof.

\begin{theorem} \label{port} At zero temperature the Jacobian of $\mu$ assumes just two values $p$ and $1-p$ almost everywhere, where $p=\frac{1+\beta_1}2$.
\end{theorem}

From now on we present some new material which was not discussed in  \cite{LMMM}.

\bigskip

Let $A$ be the set of points $\theta \in \{0,1\}^{\mathbb{N}}$ such that $J(\theta)=p$ and $B$ be the set of points $\theta\in \{0,1\}^{\mathbb{N}}$ such that $J(\theta)=1-p$. The sets $A$ and $B$ are Borel sets. Indeed,  given positive integers  $m$ and $n_0$, the set $X(m,n_0)$ of the points $\theta=(\theta_0,\theta_1,\dots)\in \{1,2\}^{\mathbb N}$, such that, for some $n\ge n_0$, $|\frac{\mu(k_0,k_1,k_2,...,k_n)}{\mu(k_1,k_2,...,k_n)}-p|>1/m$ is a union of cylinder sets, therefore an open set. From this we get $A=\cap_{m\in {\mathbb N^*}}(\{1,2\}^{\mathbb N}\setminus \cap_{n_0\in {\mathbb N^*}}X(m,n_0))$ is a Borel set. The same argument shows that $B$ is a Borel set.

\begin{figure}
\begin{center}
\includegraphics[width=8.5cm,height=6cm]{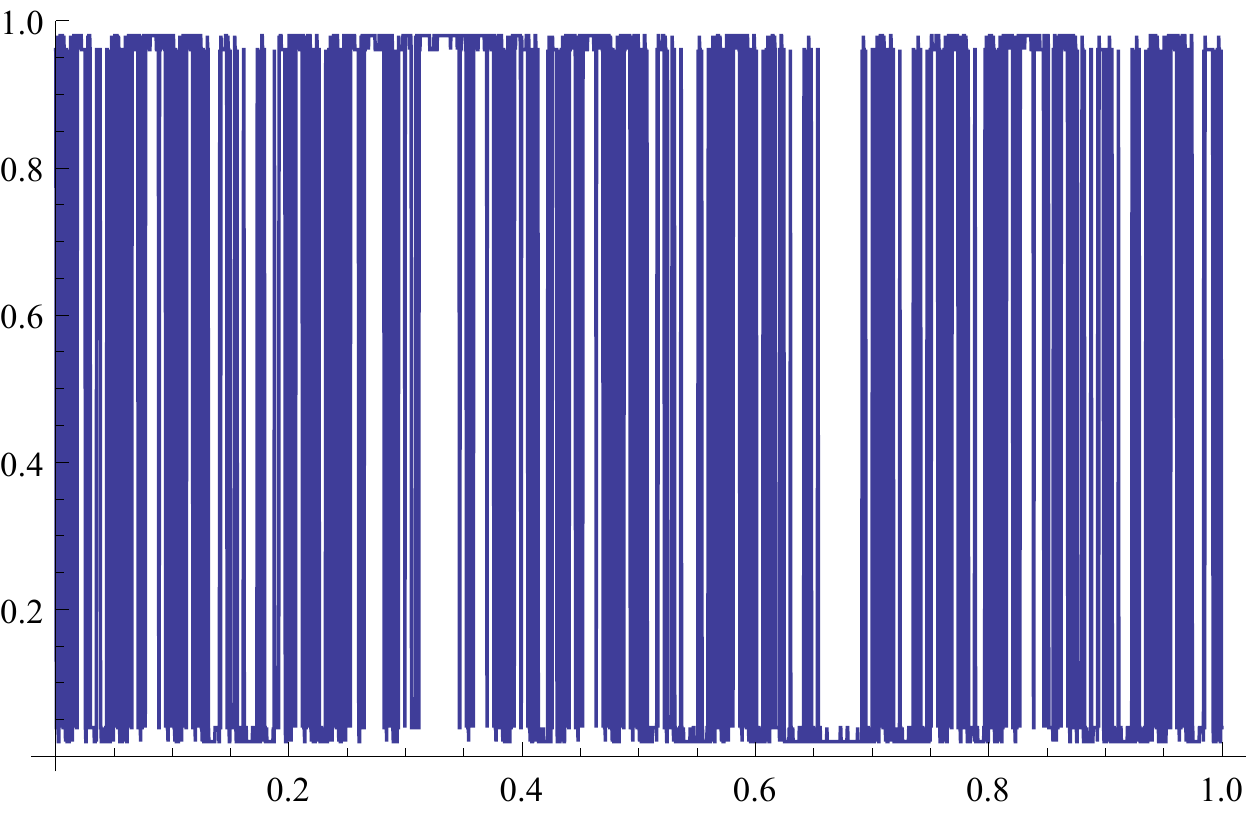}
\caption{ The graph of the values of the Jacobian $J^{12}$ at zero temperature  when $\cos (2 \, \theta) =0.96$. The maximum and minimum values are $p$ and $(1-p)$. Above the points on the interval are associated with points in $\{1,2\}^\mathbb{N}$ using the binary expansion with symbols $1$ and $2$. (by this we mean: on the binary expansion we associate $0$ to $1$, and, $1$ to $2$). We considered strings with $12$ symbols $k_0,k_1,...,k_{11}\in\,\{1,2\} $.
}
\end{center}
\end{figure}

The next proposition assures that when comparing $\theta$ and $\sigma(\theta)$ we have the alternative: there are a change of
the set (where they are), or, there are a change of the first coordinate.

\begin{proposition}
Given $\theta=(k_0,k_1,k_2,...)\in A\cup B$, the points $\theta$ and $\sigma(\theta)$ belong to the same set (both in $A$ or both in $B$), iff, $k_0\neq k_1$.
\end{proposition}

\textbf{Proof}
As $\gamma/p=1-p, \,1-\gamma/p=p, \,\gamma/(1-p)=p$ and $1-\gamma/(1-p)=1-p$, using the equation (\ref{cfg1}), we get:
\begin{itemize}
	\item $\text{ if }\theta=(k_0,k_1,k_2,...) \in A \text{ begins with }1, \text{ then }\, 2\theta \in A,\, 1\theta \in B,$
	\item $\text{ if }\theta=(k_0,k_1,k_2,...) \in A \text{ begins with }2, \text{ then }\, 1\theta \in A,\, 2\theta \in B,$
	\item $\text{ if }\tilde\theta=(k_0,k_1,k_2,...) \in B \text{  begins with }1, \text{  then }\, 1\tilde\theta \in A,\, 2\tilde\theta \in B,$
	\item $\text{if }\tilde{\theta}=(k_0,k_1,k_2,...) \in B \text{  begins with  }2,\text{  then}\, 2\tilde\theta \in A, \,1\tilde\theta \in B.$
\end{itemize}
%$$\text{ if }\theta=(k_0,k_1,k_2,...) \in A \text{ begins with }1, \text{ then }\, 2\theta \in A,\, 1\theta \in B,$$
%$$\text{ if }\theta=(k_0,k_1,k_2,...) \in A \text{  begins with }2, \text{  then } \,1\theta \in A, \,2\theta \in B,$$
%$$\text{ if }\tilde\theta=(k_0,k_1,k_2,...) \in B \text{  begins with }1, \text{  then }\, 1\tilde\theta \in A,\, 2\tilde\theta \in B,$$
%and
%$$\text{if }\tilde{\theta}=(k_0,k_1,k_2,...) \in B \text{  begins with  }2,\text{  then}\, 2\tilde\theta \in A, \,1\tilde\theta \in B.$$

\qed

\begin{theorem} The entropy of the quantum spin probability at zero temperature is
	$$ -p \log p - (1-p) \log (1-p).$$
\end{theorem}

{\bf Proof:} The entropy of $\mu$ is given by
$h(\mu) =-\int \log J d \mu = -\log(p)\mu(A) - \log(1-p)\mu(B).$ It remains to prove that $\mu(A)=p$ and $\mu(B)=1-p$.

We have $\sigma(A)=\sigma(B)=A\cup B=\{1,2\}^{\mathbb N}$ almost everywhere and $\sigma|A$ and $\sigma|B$
are injective, then, since the Jacobian of $\mu$ in $A$ is $p$ and in $B$ is $1-p$, for any measurable sets
$X\subset A$, $Y\subset B$, we get
\begin{equation}\label{eq2}
\mu(\sigma(X))=\mu(X)/p \hspace{1cm} \, \text{and} \hspace{1cm} \, \mu(\sigma(Y)) = \mu(Y)/(1-p).
\end{equation}
Particularly,
\begin{equation}\label{muA}
1= \mu(A\cup B) = \mu(\sigma(A)) = \mu(A)/p.
\end{equation}
This shows that $\mu(A)=p$ and consequently $\mu(B)=1-p$.

\qed

We will present later an ergodic conjugacy of $(\sigma,\mu)$ with another dynamical system which is ``in some way'' related with $(\sigma, m_p)$ where $m_p$ is the independent Bernoulli probability $(p, 1-p)$.

\bigskip

Before we start the proof of  Theorem \ref{port} we need an auxiliary result. Part of it can be found  in the arXiv version 1505.01305
of \cite{LMMM}.

\begin{lemma}\label{1221}
	For any $n\geq 1$ and $k_1,...,k_n \in \{1,2\}$ we have
	\[ \frac{\mu(1,2,2,1,1,k_1,...,k_n)}{\mu(2,2,1,1,k_1,...,k_n)}=\frac{\mu(1,k_1,...,k_n)}{\mu(k_1,...,k_n)}\]
and
	\[ \frac{\mu(1,1,2,2,1,k_1,...,k_n)}{\mu(1,2,2,1,k_1,...,k_n)}=\frac{\mu(1,k_1,...,k_n)}{\mu(k_1,...,k_n)}.\]
Particularly,  $J$ is not defined in $(1,1,2,2)^\infty=(1,1,2,2,1,1,2,2,1,...).$
	A similar result is true if we permute the symbols $1$ and $2$.	
\end{lemma}
\begin{proof}
From (\ref{eq1}) we get
	\[\frac{\mu(1,2,2,1,1,k_1,...,k_n)}{\mu(2,2,1,1,k_1,...,k_n)}=1-\gamma\frac{1}{0+\gamma\frac{1}{1-\gamma\frac{1}{0+\gamma \frac{1}{\frac{\mu(1,k_1,...,k_n)}{\mu(k_1,...,k_n)}}  }}}= \frac{\mu(1,k_1,...,k_n)}{\mu(k_1,...,k_n)}\]
	and
\[\frac{\mu(1,1,2,2,1,k_1,...,k_n)}{\mu(1,2,2,1,k_1,...,k_n)} = \gamma \frac{1}{1- \gamma \frac{1}{\gamma\frac{1}{1-\gamma\frac{1}{\frac{\mu(1,k_1,...,k_n)}{\mu(k_1,...,k_n)}}}   }   } = \frac{\mu(1,k_1,...,k_n)}{\mu(k_1,...,k_n)}.	\]
Particularly,
\[\frac{\mu(1,1,2,2,...,1,1,2,2,1,1)}{\mu(1,2,2,...,1,1,2,2,1,1)}=\frac{\mu(1,1)}{\mu(1)}=\frac{\gamma}{2}\]
and
\[\frac{\mu(1,1,2,2,...,1,1,2,2,1)}{\mu(1,2,2,...,1,1,2,2,1)}=\mu(1)=\frac{1}{2}.\]
This proves that $J$ is not defined in $(1,1,2,2)^\infty$.

These computations can be applying permuting the symbols $1$ and $2$.
\end{proof}

A slightly  different version of the next proof can be found  in the arXiv version of \cite{LMMM}.

\bigskip

\noindent
{\bf Proof of Theorem \ref{port}:}
	%Remember that
	
	%a) if $k_0=k_1$, then $a(k_0,k_1)=0$ and $b(k_0,k_1)=\gamma=\frac{1}{4} (1-\beta_1^2)$
	
	%b) if $k_0\neq k_1$, then $a(k_0,k_1)=1$ and $b(k_0,k_1)=-\gamma=\frac{1}{4} (\beta_1^2-1)$.
	
	\noindent
	%In this way
	The value of $J$ does not change if we permute $1$ and $2$ in the sequence $(k_0,k_1,...)\in \{1,2\}^\mathbb{N}$. For example
	\[J(1,1,1,2,2,1,...)=J(2,2,2,1,1,2,...)\]
	(if the limit exists).
	Therefore, we will introduce another code.
	%If $k_0=k_1$ then we associated  the symbol $a$ and  if $k_0\neq k_1$ we associate the symbol $b$.
	For each given sequence $k\in \{1,2\}^\mathbb{N}$, $k=(k_0,k_1,k_2,...)$ we associate a new sequence $m=m(k)=(m_0,m_1,m_2,..)\in\{a,b\}^\mathbb{N}$ by the rule: $m_i=a$ if $k_i=k_{i+1}$ and $m_i=b$ if $k_i\neq k_{i+1}$. So, we are looking if there is a change, or not, in the string $k$ by
	using the rules
	$ \underbrace{11}_a$\,, \,$ \underbrace{22}_a$\,,\,$ \underbrace{12}_b$\,,\, $ \underbrace{21}_b$.
	
	For example, given a sequence $k$ of the form
	$$ k=(1,2,1,1,2,2,... ),$$
	then, we associate
	$m=(b,b,a,b,..).$ Clearly, we can consider $J$ defined over  $\{a,b\}^\mathbb{N}$, from $J(m(k)):=J(k)$.
	\bigskip

	%Given $m$ we can find  the $k=(k_0,k_1,k_2,...)$  to whom $m$ it is associated if we know the first element $k_0=1$ or $k_0=2$.

	\medskip
	
	It can be checked that:
	\begin{align*}
	J(a,a,m_3,..) &= J(m_3,m_4,..)\\
	J(a,b,m_3,..) &= \frac{1}{ \gamma^{-1}- \frac{1}{   J(m_3,m_4,..)}}\\
	J(b,a,m_3,..) &= 1- J(m_3,m_4,..)\\
		J(b,b,m_3,..) &= 1 - \frac{1}{ \gamma^{-1}- \frac{1}{   J(m_3,m_4,..)}}.
	\end{align*}

	\medskip
	
	%For example, in terms of sequences $k$ the item
	
	%a) If $m_1=a,m_2=a$, then $J(m_1,m_2,m_3,..) = J(m_3,m_4,..)$ means
	%\medskip
	
	%I) $J(1,1,1,k_3,..) = J(1,k_3,..)$ and
	
	%II) $J(2,2,2,k_3,..) = J(2,k_3,..)$

	%\medskip

	%In this way if $k=(a_0,a_2,a_3,..) \in \{1,2\}^\mathbb{N}$ is an element where the fraction expansion $J(k)$ limit exists, then this value
	%does not change if we delete $00$ (respectively, $11$) of the sequence $k$ where appears $000$ (respectively, $111$). Indeed, this corresponds to $m_1=a=m_2$. In other words, this corresponds to a sequence $a\,a$.
		From Proposition 4.2 in \cite{LMMM}	we obtain that for $m=(m_1,...,m_k,a,a,a...)$,  $J(m)$ is not defined. For the other sequences,  the finite strings with $(aa)^{n}$, $n\in \mathbb{N}$, can be deleted (when $J$ converges). That is,
	\[J(m_0,m_1,...,m_j,a,a,m_{j+2},...)=J(m_0,m_1,...,m_j,m_{j+2},...). \]

Consider the letters of $m$ arranged in blocks of length 2,
	$$m=([m_1,m_2],[m_3,m_4],[m_5,m_6],...).$$
	As we are interested in the value of $J$, we can assume that no blocks have the form $[a,a]$ (we can delete them) and also that no blocks have the form $[b,b]$, because we can replace this one for the pair of blocks $[b,a],[a,b]$. %The same occur if $m =(m_1,...,m_k,(abab)^{\infty})$ or $m=(m_1,...,m_k,(baba)^{\infty})$.
	%This means  when considering $abab$ we get for example:
	%$$ J(1,1,2,2,1,k_5..)=J(1,k_5..),$$
	%that is we can delete $1,1,2,2$.
	%For example
	%$$ J(a,b,a,b,m_5,m_6,..) = J(m_5,m_6,..),$$
	%$$ J(a,b,a,b,a,b,a,b,m_9,m_{10},..) = J(m_9,m_{10},..),$$
	%$$ J(b,a,b,a,m_5,m_6,..) = J(m_5,m_6,..),$$
	%and
	%$$ J(m_1,m_2,a,b,a,b,m_7,m_8,..) = J( m_1,m_2,m_7,m_8,..).$$

	From now on it is natural to consider one level up of symbolic representation. We get a new code introducing a new dictionary where we associate $\alpha=[a,b]=a,b$ and $\beta=[b,a]=b,a$. %As we have seen, we can substitute $aa$ by $\emptyset$.
	%One can also substitute
	%$$bb  \,\,\,\text{by} \, \,\, \beta\,\alpha= baab.$$
	In this way, for $m=(m_0,m_1,m_2,...)=([m_0,m_1],[m_2,m_3],...)$ we associate $w=(w_0,w_1,w_2,...)$, where $w_i=\alpha$, if $[m_{2i},m_{2i+1}]=[a,b]$, and $w_i=\beta$, if $[m_{2i},m_{2i+1}]=[b,a]$.
	%So we can divide an $m$ string in blocks of two letters and make the above substitutions, obtaining
	%a new string $w=(w_1,w_2,w_3,...)\in \{\alpha,\beta\}^\mathbb{N}.$
	
	%For example we associate
	
	%$$m =([a,b],[b,a],[b,a],[a,b],...)$$ to
	%$$ w=( \alpha,\beta,\beta,\alpha,...),$$
	
	%$$m =(a,b,a,b,a,b,b,a,...)$$ to
	%$$ w=( \alpha,\alpha,\alpha,\beta,...),$$
	%and
	
	%$$m =([a,b],[a,a],[b,b],[a,b],...)$$ to
	%$$ w=( \alpha,\beta,\alpha,\alpha,...).$$
	We need to study the possible values of $J$ over $\{\alpha,\beta\}^{\mathbb{N}}$.
	First we remark that the strings $\alpha,\alpha$ and $\beta, \beta$ in $w$ correspond to the strings $a,b,a,b$ and $b,a,b,a$ in $m$, which can be deleted without changes of the value of $J$.  Indeed, as a consequence of Lemma \ref{1221} we get
	\[J(a,b,a,b,m_5,m_6,...)=J(b,a,b,a,m_5,m_6,...)=J(m_5,m_6,m_7,...).\]
	%From Lemma \ref{1221} we get that $J$ is not defined if $w=(w_1,...,w_k,\alpha,\alpha,\alpha,...)$. For any other $w$ there is some sequence $w'$ with the same value for $J$ (if it exists) and such that $w'$ does not contain $\alpha,\alpha$. After this step, the same argument can be applied for the strings $\beta,\beta$ in $w'$. The conclusion is that for any sequence where $J$ is well defined, it is equal to $J(\alpha,\beta,\alpha,\beta,\alpha,...)$ or $J(\beta,\alpha,\beta,\alpha,\beta,...)$. (ESTA CONCLUSAO EH PERIGOSA PORQUE  $J$ NAO EH CONTINUA. )
	
	%
	%Then we can suppose that  By considering the compositions of the corresponding functions, one can substitute (delete, indeed) parts of $w$ strings
	%
	
	%$$\beta\,\beta= baba  \,\,\,\text{by } \,\emptyset $$

	%$$ \alpha\,\alpha=abab  \,\,\,\text{by } \,\emptyset. $$

	%We present some examples in order to illustrate the code:  from the rules about $baba$ and $abab$ we get that in the case $J$ converges
	%$$J( \alpha,\alpha,\alpha,\beta,w_5...)= J(\alpha,\beta,w_5,...),$$
	
	%$$J( \alpha,\alpha,\alpha,\alpha,\beta,w_6,...)= J(\beta,w_6,...),$$
	
	%$$J( \beta,\beta,\beta,\beta,\alpha,w_6,...)= J(\alpha,w_6,...)$$
	%and
	%$$J(\beta,\alpha, \alpha,\alpha,\alpha,\alpha,\beta,w_8,...)= J(\beta,\alpha, \beta, w_8,...).$$

	In the finite fraction expansion of odd order of $k$ (which is associated to a certain string $m$ of even order and so to a string $w$) we can delete parts (in the $\alpha,\beta$ dictionary expansion)
	in such a way that we end up with the estimation of $J$ in a string $w$ of one of the kinds:
$(\alpha\,\beta)^n$, $(\alpha\,\beta)^n\alpha$,  $(\beta\,\alpha)^n$ or $(\beta\,\alpha)^n\beta$.
	
\medskip
	From now on, we use the above conclusion in order to determine what are the possible values of $J$. Consider the transformation
	
	$$ x\,\to\,f_1(x)= 1- \frac{1}{\gamma^{-1} - \frac {1}{x}}.$$
	The string $\beta \alpha$ means $baab$, which corresponds to
	$$J(w_1,w_2,w_3,..) = 1 - \frac{1}{ \gamma^{-1}- \frac{1}{   J(w_3,w_4,..)}}.$$
	
	Note that if the expansion $J(w_3,w_4,..)$ exists for the string $(w_3,w_4,..)$, then it also exists the one for $J(w_1,w_2,w_3,..) $.
	In this way $f_1(J (w_3,w_4,..))= J(w_1,w_2,w_3,..) .$
	
	\medskip

	Consider now $f_2$ defined by
	
	$$ x \,\to\, f_2(x)=\frac{1}{\gamma^{-1}-\frac {1}{1-x}}\,\,.$$
	The string $ \alpha\, \beta$ means $abba$, that is, it corresponds to
	$$J(w_1,w_2,w_3,..) = \frac{1}{\gamma^{-1}-\frac {1}{1-J(w_3,w_4,..)}}.$$
		In this way $f_2(J(w_3,w_4,..))= J(w_1,w_2,w_3,..).$
	
	\medskip
	
	Note that the fixed points for both functions $f_1(x)= 1- \frac{1}{\gamma^{-1} - \frac {1}{x}}$ and
	$f_2(x)=\frac{1}{\gamma^{-1}-\frac {1}{1-x}}$ are the same:  $p=\frac{1+\beta_1}{2}$ and $1-p=\frac{1-\beta_1}{2}$ (it's helpful to observe also that $\gamma=p(1-p)$ ).
	Furthermore, the interval $[1-p,p]$ is invariant by  $f_1$ and also by $f_2$. The point $p$ is a global attractor for $f_1$ in $(1-p,p]$ and $1-p$ is a global attractor for $f_2$ in $[1-p,p)$.
	
	As we have seen in Lemma \ref{jexpansion} it is natural to truncate $J(k_0,k_1,k_2,\dots)$ (at level $r$ for instance) by taking in the last position $r$, in the expansion of $J$, the value $1/2$. As $1/2  \in [1-p,p]$ (in fact is its center), and the interval $[1-p,p]$ is left invariant by the diffeomorphisms $g_{a}(x)=\frac{\gamma}{x}$ and $g_{b}(x)=1-\frac{\gamma}{x}$, when the limit exists the successive truncations should converge to $p$ or to $1-p$.
	
	Then, the only possible (convergent) values attained by the continuous fraction expansion of $J$ are $p$ or  $1-p$.
	\medskip
	
\qed

\medskip

\begin{proposition}
There exists a map
$$h : (\{1,2\}^{\mathbb{N}},\sigma,\mu )	\to (\{0,1\}^{\mathbb{N}},\sigma, \mu_p),$$
which is probability preserving, where $\mu_p$ is the Bernoulli independent probability associated to $(p,1-p)$. However, $(\{1,2\}^{\mathbb{N}},\sigma,\mu ) $ and $(\{0,1\}^{\mathbb{N}},\sigma, \mu_p)$ are not ergodically equivalent.

\end{proposition}

\begin{proof}
%Now  we can define a measurable map $h$ preserving probabilities such that is defined for $\mu$-almost every point of $\{1,2\}^{\mathbb N}$

We set  $h(\theta)=(c_0,c_1,c_2,...)$, where $c_j=0$, if $\sigma^j(\theta)\in A$, and $c_j=1$, if $\sigma^j(\theta)\in B$.
This map $h$ is defined for $\mu$-almost every point of $\{1,2\}^{\mathbb N}$ and we may to extend $h$ to $\{1,2\}^{\mathbb N}\setminus (A\cup B)$ as $(0,0,0,\dots)$. $h:(\{1,2\}^{\mathbb N},\mu)\to (\{0,1\}^{\mathbb N},\mu_p)$, where $\mu_p$ is the  Bernoulli independent probability associated to $p$ for $0$, and $1-p$ for $1$. Clearly $h\circ \sigma = \sigma \circ h$.

By the above properties of $A$ and $B$, we have the following expression for $h$: if $\alpha=(a_0,a_1,a_2,a_3,\dots)\in \{1,2\}^{\mathbb N}$, let $c_0=\chi_B(\alpha)$; then $h(\alpha)=(c_0,c_0+a_0+a_1+1\pmod 2,c_0+a_0+a_2\pmod 2,c_0+a_0+a_3+1\pmod 2,\dots).$

In order to show that $h$ is a measurable map we observe that for any cylinder set $[c_0,...,c_n] \subset \{0,1\}^{\mathbb{N}}$ we have
\[h^{-1}([c_0,...,c_n]) = h^{-1}([c_0])\cap h^{-1}(\sigma^{-1}[c_1] )\cap...\cap h^{-1}(\sigma^{-n}([c_n]))\]
\[ =h^{-1}([c_0])\cap \sigma^{-1}(h^{-1}([c_1] )\cap...\cap\sigma^{-n}( h^{-1}([c_n])), \]
which is a Borel set because, for each $i \in\{0,...,n\}$,  $h^{-1}([c_i])$ is  $A$ or $B$.

Now we will show that $\mu_p([c_0,...,c_n])= \mu(h^{-1}([c_0,...,c_n]))$, for any cylinder set $[c_0,c_1,...,c_n] \subset \{0,1\}^{\mathbb{N}}$. The proof is by induction.
For cylinders of length 1 we have,
\[  \mu_p([0]) = p  \stackrel{(\ref{muA})}{=}  \mu(A) = \mu(h^{-1}([0]))\,\,\, and \,\,\, \mu_p([1]) = 1-p = \mu(B) = \mu(h^{-1}([1])). \]
From now on we suppose that for any cylinder of length $n$ the claim is satisfied. Given a cylinder of length $n+1$ in the form $[0,c_1,...,c_n, c_{n+1}]$ we have
\[h^{-1}([0,c_1,...,c_n, c_{n+1}] ) = h^{-1}([0]\cap \sigma^{-1}([c_1,...,c_n])) = A \cap h^{-1}(\sigma^{-1}([c_1,...,c_n]))\]
\[ = A \cap \sigma^{-1}\circ h^{-1}([c_1,...,c_n])),   \]
and,  from equation (\ref{eq2})
\[\mu(h^{-1}([0,c_1,...,c_n, c_{n+1}] )) = \mu(A\cap \sigma^{-1}\circ h^{-1}([c_1,...,c_n]))\]
\[\stackrel{(\ref{eq2})}{=}  \mu(\sigma(A)\cap  h^{-1}([c_1,...,c_n]))) \cdot p
= \mu(h^{-1}([c_1,...,c_n]))) \cdot p = \mu_p([c_1,...,c_n])\cdot p\]
\[=\mu_p([0,c_1,...,c_n]). \]
The same kind of computations can be applied for a cylinder of the form $[1,c_1,...,c_n]$, which concludes the proof of the claim. The Kolmogorov extension theorem can be used to extend the result for any Borel set $X\subset \{0,1\}^{\mathbb{N}}$.

The two systems are not ergodically equivalent because the independent
Bernoulli system is mixing.

\end{proof}

If $\alpha=(a_0,a_1,a_2,a_3,\dots)\in \{1,2\}^{\mathbb N}$ and $c_0:=\chi_B(\alpha)$; then $h(\alpha)=(c_0,c_0+a_0+a_1+1\pmod 2,c_0+a_0+a_2\pmod 2,c_0+a_0+a_3+1\pmod 2,\dots).$  	
As an example, note that the restriction of  $h$ to $A\cap (\{1\}\times \{1,2\}^{{\mathbb N}^*})$ is given by

$$h(1,a_1,a_2,a_3,a_4,\dots)=$$
$$(0,a_1 \pmod 2,1+a_2\pmod 2,a_3\pmod 2,1+a_4\pmod 2,\dots)=$$
$$(0,2-a_1,a_2-1,2-a_3,a_4-1,\dots),$$
and, therefore is an homeomorphism onto its image $\{0\}\times\{0,1\}^{\mathbb{N}^*}$, with inverse map given by
$$(h|_{A\cap (\{1\}\times \{1,2\}^{{\mathbb N}^*})})^{-1}(0,c_1,c_2,c_3,c_4,\dots)=(1,2-c_1,c_2+1,2-c_3,c_4+1,\dots).$$
%Consider the homeomorphism $S:\{0\}\times\{0,1\}^{\mathbb N^*}\to \{1\}\times\{1,2\}^{\mathbb N^*}$ given by $S(0,c_1,c_2,c_3,c_4,\dots)=(1,2-c_1,c_2+1,2-c_3,c_4+1,\dots)$.
Therefore, $h(A\cap (\{1\}\times \{1,2\}^{\mathbb N^*}))$ is a Borel set. The same kind of argument can be
applied for $h(A\cap (\{2\}\times \{1,2\}^{\mathbb N^*})), h(B\cap (\{1\}\times \{1,2\}^{\mathbb N^*})) $
and $h(B\cap (\{2\}\times \{1,2\}^{\mathbb N^*}))$, which proves that $h(A\cup B)$ is a Borel set.    The image of  $h$ has full measure for $\mu_p$, because
\[ \mu_p (h(A\cup B)) = \mu(h^{-1}(h(A\cup B))) \geq \mu(A\cup B)=1.\]

This map $h$ is not an ergodic equivalence between $(\{1,2\}^{\mathbb N},\sigma,\mu)$ and
$(\{0,1\}^{\mathbb N},\sigma, \mu_p)$.  Otherwise, would be essentially a bijection,
that is, it will exist subsets of zero measure  $X$ of $(\{1,2\}^{\mathbb N},\mu)$ and $Y$ of
$(\{0,1\}^{\mathbb N},\mu_p)$, such that, $h$ restricted  to $\{1,2\}^{\mathbb N}\setminus X$
is a bijection with $\{0,1\}^{\mathbb N}\setminus Y$. This is not true in this case, because
$\theta \in A$, if and only if, $\theta^* \in A$,
where, if $\theta=(a_0, a_1, a_2,...)\in \{1,2\}^{\mathbb N}$, $\theta^*:=(3-a_0,3-a_1,3-a_2,...)$.
Therefore, $h(\theta)=h(\theta^*)$, for all $\theta\in \{1,2\}^{\mathbb N}$.

	As $h(\theta)=h(\theta^*)$ for all $\theta\in\{1,2\}^{\mathbb N}$, we get $h(\{1\}\times
\{1,2\}^{{\mathbb N}^*})=h(\{2\}\times \{1,2\}^{{\mathbb N}^*})=h(\{1,2\}^{\mathbb N})$.
On the other hand, $h(\theta)=h(\theta')$, if and only if,  $\theta'=\theta$ or $\theta'=\theta^*$.
Indeed, suppose that the first term of $\theta$ coincides with the first of $\theta'$.
As $h(\theta)=h(\theta')$, $\theta$ belongs to $A$, if and only if,  $\theta'$ belongs to $A$,
and $\sigma(\theta)$ belongs to $A$, if and only if,  $\sigma(\theta')$ also belongs to a $A$.
Therefore,  by the properties already discussed for the sets $A$ and $B$, the second terms of
$\theta$ e $\theta'$ coincide. By exchanging  $\theta$ and $\theta'$
by $\sigma(\theta)$ and $\sigma(\theta')$, we can show by
induction (using the equality $h(\sigma(\theta))=h(\sigma(\theta'))$)
that all terms of  $\theta$ and $\theta'$ coincide, that is,
we get $\theta=\theta'$.
If the first term of $\theta$ does not coincide with the first of $\theta'$,
then it coincides with the first term of $\theta^*$ and the same argument shows that in this case $\theta'=\theta^{*}$.

   \begin{proposition}
   	
   	There exists an ergodic equivalence $H$ between the shift
   	acting on  $(\{1,2\}^{\mathbb N},\mu)$ and a
   certain transformation $T$ acting in an invariant
   way on $(\{0,1\} \times
   \{0,1\}^{\mathbb N},\mu_0\times \mu_p)$, where
   $\mu_0$ is the uniform probability on  $\{0,1\}$,
   that is, such that, $\{0\}$ e $\{1\}$
   have both measure $1/2$.

   \end{proposition}

 {\bf Proof:}  We denote a point  of $\{0,1\} \times \{0,1\}^{\mathbb N}$
 by $(c_0, (c_1,c_2,...,c_n,...)),$ where $c_0\in\{0,1\}$, and $(c_1,c_2,...) \in \{0,1\}^{\mathbb{N}}$.

 Let $T: (\{0,1\} \times \{0,1\}^{\mathbb N} \to \{0,1\} \times \{0,1\}^{\mathbb N}$ be the transformation
  $$T(c_0,(c_1,c_2...))=(1-c_0,(c_2,c_3,...))=(1-c_0,\sigma(c_1,c_2,...)).$$
Observe that  $T$ preserves the probability $\mu_0\times \mu_p$ in $\{0,1\}\times\{1,2\}^{\mathbb N}$.

The ergodic equivalence $H$ between the two systems $(\{1,2\}^{\mathbb N},\sigma,\mu)$ and $(\{0,1\}\times\{0,1\}^{\mathbb{N}},T,\mu_0\times\mu_p)$  is given by $H=u\circ g$ where $g:\{1,2\}^{\mathbb N}\to \{0,1\}\times \{0,1\}^{\mathbb N}$ satisfies  $$g(a_0,a_1,a_2,...)=(2-a_0,h(a_0,a_1,a_2,...))$$ and $u: \{0,1\}\times \{0,1\}^{\mathbb N}\to\{0,1\}\times \{0,1\}^{\mathbb N}$ satisfies $$u(c_0,(c_1,c_2,c_3,...))=(c_0+c_1 \pmod 2,(c_1,c_2,c_3,...)).$$

From the previous discussion the transformation $g$ is injective and its image, which is $\{0,1\}\times h(\{1,2\}^{\mathbb N})$, has full measure in $\{0,1\}\times\{0,1\}^{\mathbb N}$ with respect to $\mu_0\times \mu_p$.  Then, we can consider the application $g^{-1}$. We have that $g$ is measurable and following the above discussions, the restrictions $g|_A$ and $g|_B$ are homeomorphisms onto your images. Therefore $g^{-1}$ is measurable.
Moreover, $M:=g\circ \sigma\circ g^{-1}$ is given by
$$M(c_0,(c_1,c_2,...))=(c_0+c_1+c_2+1\pmod 2,(c_2,c_3,c_4,...)).$$
Indeed, if $g(a_0,a_1,...)=(c_0,(c_1,c_2,...))$, we get
\[ g\circ \sigma\circ g^{-1}(c_0,(c_1,c_2,...)) =g\circ \sigma(a_0,a_1,...) = g(a_1,a_2,...) = (2-a_1,h(a_1,a_2,...))   \]
\[= (2-a_1,h(\sigma (a_0,a_1,a_2,...))= (2-a_1,\sigma (h(a_0,a_1,a_2,...)) \]
\[=(2-a_1,\sigma (c_1,c_2,c_3,...))= (2-a_1,(c_2,c_3,c_4,...)) \]
Now we will show that $2-a_1 = c_0+c_1+c_2+1\pmod 2$. By definition of $h$, we get $c_1=0$ if $(a_0,a_1,...)\in A$ and $c_1=1$ if  $(a_0,a_1,...)\in B$, $c_2=0$ if $(a_1,a_2,...)\in A$, and $c_2=1$ if $(a_1,a_2,...)\in B$. As $(a_0,a_1,...)$ and $(a_1,a_2,...)$ both belong to  $A$, or both belong to $B$, if and only if, $a_0\neq a_1$, it follows that $2-a_1 = a_1 \pmod 2 =a_0+c_1+c_2+1 \pmod 2 = c_0+c_1+c_2+1\pmod 2$.

Observe now that $u$ is an involution. %) given by $$u(c_0,c_1,c_2,c_3,...)=(c_0+c_1 \pmod 2,c_1,c_2,c_3,...).$$
Furthermore
$$u\circ M \circ u^{-1}(c_0,(c_1,c_2,c_3,...))=u(M(c_0+c_1 \pmod 2,(c_1,c_2,c_3,...))$$
$$=u(c_0+c_2+1 \pmod 2,(c_2,c_3,...))=(c_0+1 \pmod 2,(c_2,c_3,...))$$
$$=(1-c_0,(c_2,c_3,...))=T(c_0,(c_1,c_2,...)).$$

Therefore, as $H=u\circ g$, we get
$$H\circ \sigma \circ H^{-1}=u\circ M \circ u^{-1}=T.$$

As
$\mu(\{1,2\}^\mathbb{N} )=1$, $\mu(1)=\mu(2)=1/2$ and, for $n\ge 1$,
$$\mu(k_0,k_1,k_2,...,k_n)=a(k_0,k_1)\mu(k_1,k_2,k_3,...,k_n)+b(k_0,k_1)\mu(k_2,k_3,...,k_n),$$
we get, for any cylinder $[k_0,...,k_n]$,
\[\mu(k_0,k_1,k_2,...,k_n) = \mu(3-k_0,3-k_1,...,3-k_n),\]
and, consequently, for any Borel set $X\subset \{1,2\}^{\mathbb{N}}$ we get $\mu(X) = \mu(X^{*})$, where
$X^*=\{\theta^*|\theta \in X \}$.

It follows that for any cylinder $[c_1,c_2,...,c_n] \subset \{0,1\}^{\mathbb{N}}$,
\[\mu(H^{-1}(0, [c_1,c_2,...,c_n])) = \mu ([2]\cap h^{-1}([c_1,...,c_n]) ) = \frac{1}{2}\mu( h^{-1}([c_1,...,c_n]) )  \]
and
\[\mu(H^{-1}(1, [c_1,c_2,...,c_n])) = \mu ([1]\cap h^{-1}([c_1,...,c_n]) ) = \frac{1}{2}\mu( h^{-1}([c_1,...,c_n]) )) . \]

As
\[\frac{1}{2}\mu( h^{-1}([c_1,...,c_n]) )) = \frac{1}{2}\mu_p[c_1,...,c_n],\]
we get that $H^{*}(\mu) = \mu_0\times \mu_p$.
This proves that  $H$ is an ergodic equivalence between  the shift $\sigma$ acting on $(\{1,2\}^{\mathbb N},\mu)$ and the transformation $T$ acting on $(\{0,1\}^{\mathbb N}, \mu_0\times \mu_p)$.

\qed

In \cite{LMMM} it is proved that $\mu$ is ergodic but not mixing for $\sigma$. Now we can conclude that it is not ergodic for $\sigma^2$.

\begin{corollary}
$\mu$ is ergodic for $\sigma$ but it is not ergodic for $\sigma^2$. Particularly it is not mixing.
\end{corollary}

\begin{proof}
The measure $\mu_0\times \mu_p$ is ergodic for the map $T$  but not for  $T^2$ because it leave invariant the sets $\{0\}\times \{0,1\}^{{\mathbb N}^*}$ and $\{1\}\times \{0,1\}^{{\mathbb N}^*}$, which are permuted by $T$. Therefore, $(T, \mu_0\times \mu)$ is not mixing.
\end{proof}

The above reasoning provides another proof of that the Kolmogorov  entropy of the dynamical system $(\sigma,\mu)$ is $-p\log p-(1-p)\log(1-p)$ (see Theorem 4.23 in \cite{Wal}). That is, this entropy is equal to the entropy of the  Bernoulli shift $(\sigma,\mu_p)$.

\section{Appendix} \label{app}

			\begin{proposition}\label{hipotese_kolmogorov1}
				For all $n\in \{1,2,3,...\}$ we get
				\begin{equation*}
				\mu_{\beta,n+1}(j_1, ..., j_n, 1) + \mu_{\beta,n+1}(j_1, ..., j_n, 2) = \mubn(j_1, ..., j_n).
				\end{equation*}
			\end{proposition}
			\begin{proof}  For $n=1$ the result can be checked explicitly from example \ref{ex_mu2emu3}. For $n\geq 2$ we use  (\ref{def_mubn}):
				
				$$
				\mu_{\beta,n+1}(j_1, ..., j_n, 1) + \mu_{\beta,n+1}(j_1, ..., j_n, 2) =$$
$$\frac{1}{2^{n+1}}\tr\left[\begin{array}{lll}
					 \prod_{i=1}^{n} [ I^{\otimes\, n+1}  \,+\, (-1 \,+\, \phib)\, (\sigma^x_{i}\otimes \sigma^x_{i+1})_{n+1}]\,\circ\\
					 \left(P_{j_1} \otimes  P_{j_2}\,\otimes...\otimes P_{j_n} \,\otimes\, P_1\right)
				\end{array}\right] 	$$
$$
				+\frac{1}{2^{n+1}}\tr\left[\begin{array}{lll}
					 \prod_{i=1}^{n} [ I^{\otimes\, n+1}  \,+\, (-1 \,+\, \phib)\, (\sigma^x_{i}\otimes \sigma^x_{i+1})_{n+1}]\,\circ\\
					 \left( P_{j_1} \otimes  P_{j_2}\,\otimes...\otimes P_{j_n} \,\otimes\, P_2\right)
				\end{array}\right]  	$$
$$=\frac{1}{2^{n+1}}\tr\left[\begin{array}{lll}
					 \prod_{i=1}^{n}[ I^{\otimes\, n+1}  \,+\, (-1 \,+\, \phib)\, (\sigma^x_{i}\otimes \sigma^x_{i+1})_{n+1}]\,\circ\\
					\left( P_{j_1} \otimes  P_{j_2}\,\otimes...\otimes P_{j_n} \otimes \, (P_1+P_2)\right)
				\end{array}\right]= $$
$$		
		\frac{1}{2^{n+1}}\tr\left[\begin{array}{lll}
					\prod_{i=1}^{n}[ I^{\otimes\, n+1}  \,+\, (-1 \,+\, \phib)\, (\sigma^x_{i}\otimes \sigma^x_{i+1})_{n+1}]\,\circ\\
					\left( P_{j_1} \otimes  P_{j_2}\,\otimes...\otimes P_{j_n}\, \otimes \, I\right)
				\end{array}\right] =$$ %%		
$$\frac{1}{2^{n+1}}\tr\left[\begin{array}{lll}
					 \prod_{i=1}^{n-1} [ I^{\otimes\, n+1}  \,+\, (-1 \,+\, \phib)\, (\sigma^x_{i}\otimes \sigma^x_{i+1})_{n+1}]\,\circ\\
					 \left( P_{j_1} \otimes  P_{j_2}\,\otimes...\otimes P_{j_n}\, \otimes \, I\right)
				\end{array}\right] +  $$		

$$\frac{(-1 \,+\, \phib)}{2^{n+1}}\tr\left[\begin{array}{lll}
				\prod_{i=1}^{n-1} [ I^{\otimes\, n+1}  \,+\, (-1 \,+\, \phib)\, (\sigma^x_{i}\otimes \sigma^x_{i+1})_{n+1}]\,\circ\\
				\left( P_{j_1} \otimes  P_{j_2}\,\otimes...\otimes P_{j_n}\, \otimes \, I\right)\,\circ \\
				(\sigma^x_{n}\otimes \sigma^x_{n+1})_{n+1}
				\end{array}
				\right].
				$$

%\begin{figure}
%\begin{center}
%\includegraphics[width=10.5cm,height=8cm]{fi-pos2.pdf}
%\caption{ The graphs of $f$, the identity and the constant function equal to $\alpha$, when $\alpha= 1/5$ and $\gamma= 1/9$.}
%\end{center}
%\end{figure}

%\begin{figure}
%\begin{center}
%\includegraphics[width=10.5cm,height=8cm]{fi-pos1.pdf}
%\caption{ The graphs of $g$, the identity and the constant function equal to $1-\alpha$, when $\alpha= 1/5$ and $\gamma= 1/9$.}
%\end{center}
%\end{figure}

				\medskip
				
				 As $\tr(L_1\otimes\,\cdots\,\otimes L_n) = \tr(L_1)\cdots\tr(L_n)$ and $\tr(\sigma_x) = 0$, we finally  get:
				$$
				\mu_{\beta,n+1}(j_1, ..., j_n, 1) + \mu_{\beta,n+1}(j_1, ..., j_n, 2) =$$
$$\frac{1}{2^{n+1}}\tr\left[\begin{array}{lll}
					\prod_{i=1}^{n-1} \left[I^{\otimes n+1} \,+\, (-1 \,+\, \phib)\,(\sigma^x_{i}\otimes \sigma^x_{i+1})_{n+1}\right]\,\circ\\
				\left( P_{j_1} \otimes  P_{j_2}\,\otimes...\otimes P_{j_n}\, \otimes \, I\right)
				\end{array}\right].
				$$

Note that the last term in each term of the above  tensor products expression is the identity. As  $\tr(I) = 2$, then
				
				$$
				\mu_{\beta,n+1}(j_1, ..., j_n, 1) + \mu_{\beta,n+1}(j_1, ..., j_n, 2) =$$
				$$\frac{1}{2^{n}}\tr\left[\begin{array}{lll}
					\prod_{i=1}^{n-1} \left[I^{\otimes n} \,-\, (\sigma^x_{i}\otimes \sigma^x_{i+1})_{n} \,+\, \phib\, (\sigma^x_{i}\otimes \sigma^x_{i+1})_{n}\right]\,\circ\\
					\left( P_{j_1} \otimes  P_{j_2}\,\otimes...\otimes P_{j_n} \right)
				\end{array}\right]=$$
$$\mubn(j_1, ..., j_n).
				$$
			\end{proof}
			
\medskip

		\begin{theorem}\label{medidacilindro1}
			For the probability  $\mub$ and for any $n\ge 2$, we get
			$$
			\mub(k,j_1,...,j_n)= $$
			$$\frac{\mub( j_1,...,j_n )}{2}  + \sum_{i=1}^{n-2}\frac{(-1+\phib)^i \beta_k \beta_{j_i} }{2^{i+1}}\mub(j_{i+1},...,j_n)  +$$
$$\frac{(-1+\phib)^{n-1}\beta_k\beta_{j_{n-1}}}{2^{n+1}}+ \frac{(-1+\phib)^{n}\beta_k\beta_{j_n}}{2^{n+1}}.
			$$
		\end{theorem}
		
		\begin{proof} Using equation (\ref{def_mubn}) and the fact that $\mub$ coincide with $\mubn$ in cylinders, we get:
			
			$$
				\mub(k,j_1,...,j_n) = \frac{1}{2^{n+1}}\tr\left[\begin{array}{ll}
						\prod_{i=1}^{n} [I^{\otimes n+1}\, + \,(-1 \, + \, \phib\,)\, (\sigma^x_{i}\otimes \sigma^x_{i+1})_{n+1}]\,\circ\\
						(   P_k \otimes P_{j_1} \otimes  P_{j_2}\,\otimes...\otimes P_{j_n})	
					\end{array}\right]=$$
				$$  \frac{1}{2^{n+1}}\tr\left[\begin{array}{ll}
					\prod_{i=2}^{n} [I^{\otimes n+1}\, + \,(-1 \, + \, \phib\,)\, (\sigma^x_{i}\otimes \sigma^x_{i+1})_{n+1}]\,\circ\\
				(P_k \otimes P_{j_1} \otimes  P_{j_2}\,\otimes...\otimes P_{j_n})
				\end{array}\right]+$$
$$ \frac{(-1 \, + \, \phib\,)}{2^{n+1}}\tr\left[\begin{array}{ll}
					\prod_{i=2}^{n} [I^{\otimes n+1}\, + \,(-1 \, + \, \phib\,)\, (\sigma^x_{i}\otimes \sigma^x_{i+1})_{n+1}]\,\circ\\
					(\sigma^x_{1}\otimes \sigma^x_{2})_{n+1}\,\circ\\
					(P_k \otimes P_{j_1} \otimes  P_{j_2}\,\otimes...\otimes P_{j_n})
				\end{array}\right]=$$
				$$\frac{1}{2^{n+1}}\tr\left[\begin{array}{ll}
				 \prod_{i=1}^{n-1} [I^{\otimes n}\, + \,(-1 \, + \, \phib\,)\, (\sigma^x_{i}\otimes \sigma^x_{i+1})_{n}]\,\circ\\
					(P_{j_1} \otimes  P_{j_2}\,\otimes...\otimes P_{j_n})
				\end{array}\right]+$$

	$$			 \frac{(-1 \, + \, \phib\,)}{2^{n+1}}\tr\left[\begin{array}{ll}
					 \prod_{i=2}^{n} [I^{\otimes n+1}\, + \,(-1 \, + \, \phib\,)\, (\sigma^x_{i}\otimes \sigma^x_{i+1})_{n+1}]\,\circ\\
					(\sigma^x_{1}\otimes \sigma^x_{2})_{n+1}\,\circ\\
					(P_k \otimes P_{j_1} \otimes  P_{j_2}\,\otimes...\otimes P_{j_n})
				\end{array}\right]=$$
				$$
					\frac{1}{2}\mub(j_1,...,j_n) + $$
$$
				 \frac{(-1 \, + \, \phib\,)}{2^{n+1}}\tr\left[\begin{array}{ll}
					\prod_{i=2}^{n} [I^{\otimes n+1}\, + \,(-1 \, + \, \phib\,)\, (\sigma^x_{i}\otimes \sigma^x_{i+1})_{n+1}] \,\circ\\
					(\sigma^x_{1} P_k \otimes \sigma^x_{2} P_{j_1} \otimes  P_{j_2}\,\otimes...\otimes P_{j_n})
				\end{array}\right]=$$
			$$ \frac{1}{2}\mub(j_1,...,j_n) + $$
$$ \frac{(-1 \, + \, \phib\,)}{2^{n+1}}\tr\left[\begin{array}{ll}
					\prod_{i=3}^{n} [I^{\otimes n+1}\, + \,(-1 \, + \, \phib\,)\, (\sigma^x_{i}\otimes \sigma^x_{i+1})_{n+1}] \,\circ\\
					(\sigma^x_{1} P_k \otimes \sigma^x_{2} P_{j_1} \otimes  P_{j_2}\,\otimes...\otimes P_{j_n})
				\end{array}\right]+$$
$$\frac{(-1 \, + \, \phib\,)^2}{2^{n+1}}\tr\left[\begin{array}{lll}
				\prod_{i=3}^{n} [I^{\otimes n+1}\, + \,(-1 \, + \, \phib\,)\, (\sigma^x_{i}\otimes \sigma^x_{i+1})_{n+1}] \,\circ\\(\sigma^x_{2}\otimes \sigma^x_{3})_{n+1} \,\circ\\
			(\sigma^x_{1} P_k \otimes \sigma^x_{2} P_{j_1} \otimes  P_{j_2}\,\otimes...\otimes P_{j_n})
				\end{array}\right]=$$
		$$\frac{1}{2}\mub(j_1,...,j_n) + \frac{(-1 \, + \, \phib\,)\beta_{k}\beta_{j_1}}{2^2}\mub(j_2,...,j_n)\, +$$
			$$ \frac{(-1 \, + \, \phib\,)^2}{2^{n+1}}\tr\left[\begin{array}{ll}
					\prod_{i=3}^{n} [I^{\otimes n+1}\, + \,(-1 \, + \, \phib\,)\, (\sigma^x_{i}\otimes \sigma^x_{i+1})_{n+1}]\,\circ\\
					(\sigma^x_{2}\otimes \sigma^x_{3})_{n+1}\,\circ\\
					(\sigma^x_{1} P_k \otimes \sigma^x_{2} P_{j_1} \otimes  P_{j_2}\,\otimes...\otimes P_{j_n})
				\end{array}\right] =
			$$
			
			$$\vdots$$
			$$
				\frac{1}{2}\mub(j_1,...,j_n) + \frac{(-1 \, + \, \phib\,)\beta_{k}\beta_{j_1}}{2^2}\mub(j_2,...,j_n)\, +
				 ... +$$
$$ \frac{(-1 \, + \, \phib\,)^{n-2}\beta_{k}\beta_{j_{n-2}}}{2^{n-1}}\mub(j_{n-1},j_n) +$$
				$$\frac{(-1 \, + \, \phib\,)^{n-1}}{2^{n+1}}\tr\left[\begin{array}{lll}
					 (I^{\otimes n+1}\, + \,(-1 \, + \, \phib\,)(\sigma^x_{n}\otimes \sigma^x_{n+1})_{n+1})\,\circ\\
					 (  \sigma^x_{1} P_k \otimes  P_{j_1} \otimes  P_{j_2}\,\otimes...\otimes \sigma^x_{n}P_{j_{n-1}}\otimes P_{j_n})\end{array}\right]  =	$$
$$ \frac{\mub( j_1,...,j_n )}{2}  + \sum_{i=1}^{n-2}\frac{(-1+\phib)^i \beta_k \beta_{j_i} }{2^{i+1}}\mub(j_{i+1},...,j_n)  +$$
$$\frac{(-1+\phib)^{n-1}\beta_k\beta_{j_{n-1}}}{2^{n+1}} + \frac{(-1+\phib)^{n}\beta_k\beta_{j_n}}{2^{n+1}}.
			$$

\medskip
Above we use several times the property $\tr(L\otimes\,\cdots\,\otimes L) = \tr(L)\cdots\tr(L)$ and
the linearity of the trace.

		\end{proof}

\medskip

		\begin{proposition}\label{recursive1}
			For any $n\geq 1$, we get
			$$
				\mub (k_0 , k_{1} ,...,k_n)=$$
$$\frac{1}{2}\bigg(1+\frac{\beta_{k_0}}{\beta_{k_1}}(-1+\phib)\bigg)\mub (k_1 , ...,k_{n})\,+
										   \frac{\beta_{k_0}}{2}(-1+\phib) \bigg(\frac{-1}{2\beta_{k_{1}}}+\frac{\beta_{k_{1}}}{2}\bigg)\mub (k_2,...,k_{n}).
			$$

		\end{proposition}
		
		\begin{proof}
			The cases $n=1,2$ correspond to
			\noindent\[ \mub (k_0 , k_{1})=\frac{1}{2}\bigg(1+\frac{\beta_{k_0}}{\beta_{k_1}}(-1+\phib)\bigg)\mub (k_1) +\frac{\beta_{k_0}}{2}(-1+\phib) \bigg(\frac{-1}{2\beta_{k_{1}}}+\frac{\beta_{k_{1}}}{2}\bigg)\]
			\noindent and
			
$$\mub (k_0 , k_1 ,k_2)=\frac{1}{2}\bigg(1+\frac{\beta_{k_0}}{\beta_{k_1}}(-1+\phib)\bigg)\mub (k_1 , k_2) +$$
$$\frac{\beta_{k_0}}{2}(-1+\phib) \bigg(\frac{-1}{2\beta_{k_{1}}}+\frac{\beta_{k_{1}}}{2}\bigg)\mub (k_2) ,$$
			\noindent which can be directly obtained by using the fact that $\mub(k)=1/2, \,k=1,2$.
			
			 For the case $n\geq 2$, note that from Theorem  \ref{medidacilindro} we get the equations
			$$
			\frac{2\mub(k_0,k_1,...,k_n)}{\beta_{k_0}} = \frac{\mub(k_1,...,k_n)}{\beta_{k_0}} + \sum_{i=1}^{n-2}\frac{(-1+\phib)^i\beta_{k_i}\mub(k_{i+1},...,k_n)}{2^i}\,+$$
$$ \frac{(-1+\phib)^{n-1}\beta_{k_{n-1}}}{2^n} + \frac{(-1+\phib)^n\beta_{k_n}}{2^n} =$$
	$$	\frac{\mub(k_1,...,k_n)}{\beta_{k_0}} + \frac{(-1+\phib)\beta_{k_1}\mub(k_2,...,k_n)}{2}\,+$$
			$$\sum_{i=2}^{n-2}\frac{(-1+\phib)^i\beta_{k_i}\mub(k_{i+1},...,k_n)}{2^i}\,+$$
			$$\frac{(-1+\phib)^{n-1}\beta_{k_{n-1}}}{2^n} + \frac{(-1+\phib)^n\beta_{k_n}}{2^n},$$
			
			\noindent and
			$$
			\frac{(-1+\phib)\mub(k_1,...,k_n)}{\beta_{k_1}} = \frac{(-1+\phib)\mub(k_2,...,k_n)}{2\beta_{k_1}}\,+$$ $$\sum_{i=1}^{n-3}\frac{(-1+\phib)^{i+1}\beta_{k_{i+1}}\mub(k_{i+2},...,k_n)}{2^{i+1}}\,+$$
$$ \frac{(-1+\phib)^{n-1}\beta_{k_{n-1}}}{2^{n}} + \frac{(-1+\phib)^{n}\beta_{k_n}}{2^{n}} =$$
		$$\frac{(-1+\phib)\mub(k_2,...,k_n)}{2\beta_{k_1}}\, +$$
			$$\sum_{i=2}^{n-2}\frac{(-1+\phib)^{i}\beta_{k_{i}}\mub(k_{i+1},...,k_n)}{2^{i}}\,+$$
			$$ \frac{(-1+\phib)^{n-1}\beta_{k_{n-1}}}{2^{n}} + \frac{(-1+\phib)^{n}\beta_{k_n}}{2^{n}}.
			$$
			 Then,
			$$
			\frac{2\mub(k_0,k_1,...,k_n)}{\beta_{k_0}} - \frac{(-1+\phib)\mub(k_1,...,k_n)}{\beta_{k_1}}$$ $$
			= \left(\frac{\mub(k_1,...,k_n)}{\beta_{k_0}} + \frac{(-1+\phib)\beta_{k_1}\mub(k_2,...,k_n)}{2}\right) - \left( \frac{(-1+\phib)\mub(k_2,...,k_n)}{2\beta_{k_1}}  \right).
			$$ 	
			 Therefore,
			\begin{align*}
			\mub(k_0,k_1,...,k_n) &= \frac{1}{2}\bigg( 1 + \frac{\beta_{k_0}}{\beta_{k_1}}(-1+\phib)\bigg)\mub(k_1,...,k_n) + \\ &\frac{\beta_{k_0}}{2}(-1+\phib)\bigg(-\frac{1}{2\beta_{k_1}} + \frac{\beta_{k_1}}{2}\bigg)\mub(k_2,...,k_n).
			\end{align*}
			
		\end{proof}
		
\medskip

		\begin{theorem}\label{teoQ1} For all $n\in \mathbb{N}$ and $t\in\mathbb{R}$
			
			\begin{equation}\label{main1}
			Q_n (t)=\left[\begin{array}{l}\frac{1}{2}\delta(t)\,Q_{n-1}(t) + \frac{(-1+\phib)}{4}\alpha(t)^2\,Q_{n-2}(t)+ \frac{(-1+\phib)^2}{8}\delta(t)\,\alpha(t)^2\, Q_{n-3}(t) \\ \\
			\,\,+  \frac{(-1+\phib)^3}{16}\delta(t)^2\,\alpha(t)^2\, \, Q_{n-4}(t)+\frac{(-1+\phib)^4}{32}\delta(t)^3\,\alpha(t)^2\, \, Q_{n-5}(t)+...+\\  \\      +\frac{(-1+\phib)^{n-3}}{2^{n-2}}\, \delta(t)^{n-4}\,\alpha(t)^2\, \, Q_{ 2}(t)
			+ \frac{(-1+\phib)^{n-2}}{2^{n-1}}\, \delta(t)^{n-3}\,\alpha(t)^2\, \, Q_{ 1}(t)\\ \\
			+ \frac{\phib(-1+\phib)^{n-1}}{2^{n+1}}\alpha^2(t)\delta^{n-2}(t).\end{array}\right] .
			\end{equation}
			
		\end{theorem}

		\begin{proof}
			By definition
			$$Q_n(t) =  \sum_{j_0}\, \sum_{j_1}\, ...\sum_{j_n} e^{t\, (  A(j_0) + A(j_1) +...+ A(j_n) ) }\, \mub ( j_0,j_1,...j_n),$$
			\noindent\2 and by Theorem \ref{medidacilindro}			
			\begin{align*}
				\mub (&j_0 , j_1 , j_2,...,j_n) \,=\,\,\, \frac{\mub( j_1,j_2,...,j_n )}{2} + \frac{(-1+\phib)\beta_{j_0}\, \beta_{j_1}}{2^{2}} \, \mub (j_2,...,j_n) + \,\\
				&\frac{(-1+\phib)^2\beta_{j_0}\beta_{j_2}}{2^{3}} \mub( j_3,j_4,...,j_n ) + \frac{(-1+\phib)^3\beta_{j_0} \beta_{j_3}}{2^{4}} \mub( j_4,j_5,...,j_n )+...+\\
				&\frac{(-1+\phib)^{n-2} \beta_{j_0} \beta_{j_{n-2}}}{2^{n-1}} \mub(j_{n-1},j_n)+ \frac{(-1+\phib)^{n-1} \beta_{j_0} \beta_{j_{n-1}}}{2^{n+1}}+ \frac{(-1+\phib)^{n} \beta_{j_0} \beta_{j_n}}{2^{n+1}}.
			\end{align*}

			\2 Therefore,
			$$
			Q_n(t)=\frac{1}{2}\sum_{j_0}e^{t\,  A(j_0)}\sum_{j_1,...,j_n} e^{t\, ( A(j_1) +...+ A(j_n) ) }\, \mub (j_1,...j_n)+
			$$
			$$+\frac{(-1+\phib)}{2^2}\sum_{j_0}e^{t\,  A(j_0)}\beta_{j_0}\sum_{j_1}e^{tA(j_1)}\beta_{j_1}\sum_{j_2,...,j_n} e^{t\, ( A(j_2) +...+ A(j_n) ) }\, \mub (j_2,...j_n)$$
			$$+\frac{(-1+\phib)^2}{2^3}\sum_{j_0}e^{tA(j_0)}\beta_{j_0}\sum_{j_1}e^{tA(j_1)}\sum_{j_2}e^{tA(j_2)}\beta_{j_2}\sum_{j_3,...,j_n} e^{t\, ( A(j_3) +...+ A(j_n) ) }\, \mub (j_3,...j_n)$$
			$$...+ \frac{(-1+\phib)^{n}}{2^{n+1}}  \sum_{j_0}e^{tA(j_0)}\beta_{j_0} \sum_{j_n}e^{tA(j_n)}\beta_{j_n}\sum_{j_2}e^{tA(j_2)}...\sum_{j_{n-1}}e^{tA(j_{n-1})}=$$
		
			\begin{align*}
				&\frac{1}{2}\delta(t)\,Q_{n-1}(t)+ \frac{(-1+\phib)}{4}\alpha(t)^2\,Q_{n-2}(t)+ \frac{(-1+\phib)^2}{8}\delta(t)\,\alpha(t)^2\, Q_{n-3}(t)\,+\\  &\hspc\frac{(-1+\phib)^3}{16}\delta(t)^2\,\alpha(t)^2\, \, Q_{n-4}(t)+
				\frac{(-1+\phib)^4}{32}\delta(t)^3\,\alpha(t)^2\, \, Q_{n-5}(t)+...      + \\
				&\hspc\frac{(-1+\phib)^{n-3}}{2^{n-2}}\, \delta(t)^{n-4}\,\alpha(t)^2\, \, Q_{ 2}(t)
				+ \frac{(-1+\phib)^{n-2}}{2^{n-1}}\, \delta(t)^{n-3}\,\alpha(t)^2\, \, Q_{ 1}(t)\,+\\
				&\hspc\frac{(-1+\phib)^{n-1}}{2^{n+1}}\alpha(t)^2\delta(t)^{n-2}+ \frac{(-1+\phib)^{n}}{2^{n+1}}\alpha(t)^2\delta(t)^{n-2}\\
				&=\frac{1}{2}\delta(t)\,Q_{n-1}(t)+ \frac{(-1+\phib)}{4}\alpha(t)^2\,Q_{n-2}(t)+ \frac{(-1+\phib)^2}{8}\delta(t)\,\alpha(t)^2\, Q_{n-3}(t)\,+\\  &\hspc\frac{(-1+\phib)^3}{16}\delta(t)^2\,\alpha(t)^2\, \, Q_{n-4}(t)+
				\frac{(-1+\phib)^4}{32}\delta(t)^3\,\alpha(t)^2\, \, Q_{n-5}(t)+...      +\\
				&\hspc\frac{(-1+\phib)^{n-3}}{2^{n-2}}\, \delta(t)^{n-4}\,\alpha(t)^2\, \, Q_{ 2}(t)
				+ \frac{(-1+\phib)^{n-2}}{2^{n-1}}\, \delta(t)^{n-3}\,\alpha(t)^2\, \, Q_{ 1}(t)\,+\\ &\hspc\frac{\phib(-1+\phib)^{n-1}}{2^{n+1}}\alpha^2(t)\delta^{n-2}(t).
			\end{align*}			
			
		\end{proof}

\medskip

		\begin{proposition}\label{recurrence_relation1}
			\[Q_{n+2}(t) = (-1+\phib)\frac{\alpha(t)^2-\delta(t)^2}{4}\,Q_{n}(t) + \frac{\phib}{2}\delta(t)\,Q_{n+1}(t).\]
		\end{proposition}

		\begin{proof}
			
		From the above reasoning we get
			
			\begin{align*}
				Q_n(t)=&\frac{1}{2}\delta(t)\,Q_{n-1}(t)+ \frac{(-1+\phib)}{4}\alpha(t)^2\,Q_{n-2}(t)+ \frac{(-1+\phib)^2}{8}\delta(t)\,\alpha(t)^2\, Q_{n-3}(t)\,+\\
				&\frac{(-1+\phib)^3}{16}\delta(t)^2\,\alpha(t)^2\, \, Q_{n-4}(t)+
				\frac{(-1+\phib)^4}{32}\delta(t)^3\,\alpha(t)^2\, \, Q_{n-5}(t)+...+\\
				&\frac{(-1+\phib)^{n-3}}{2^{n-2}}\, \delta(t)^{n-4}\,\alpha(t)^2\, \, Q_{ 2}(t)
				+ \frac{(-1+\phib)^{n-2}}{2^{n-1}}\, \delta(t)^{n-3}\,\alpha(t)^2\, \, Q_{ 1}(t) \,+ \\ 	
				&\frac{\phib(-1+\phib)^{n-1}}{2^{n+1}}\alpha^2(t)\delta^{n-2}(t),
			\end{align*}
			
			\noindent and using again the last proposition for $Q_{n-1}(t)$ we get
			\begin{align*}
				Q_{n-1}(t)=&\frac{1}{2}\delta(t)\,Q_{n-2}(t)+ \frac{(-1+\phib)}{4}\alpha(t)^2\,Q_{n-3}(t)\,+\\
				&\frac{(-1+\phib)^2}{8}\delta(t)\,\alpha(t)^2\, Q_{n-4}(t)+  \frac{(-1+\phib)^3}{16}\delta(t)^2\,\alpha(t)^2\, \, Q_{n-5}(t)\,+\\
				&\frac{(-1+\phib)^4}{32}\delta(t)^3\,\alpha(t)^2\, \, Q_{n-6}(t)+...+ \frac{(-1+\phib)^{n-4}}{2^{n-3}}\, \delta(t)^{n-5}\,\alpha(t)^2\, \, Q_{ 2}(t)
				\,+\\
				&\frac{(-1+\phib)^{n-3}}{2^{n-2}}\, \delta(t)^{n-4}\,\alpha(t)^2\, \, Q_{ 1}(t)
				+ \frac{\phib(-1+\phib)^{n-2}}{2^{n}}\alpha^2(t)\delta^{n-3}(t).
			\end{align*}
			
			\noindent Therefore,
			\begin{align*}
				Q_n(t)  - \frac{(-1+\phib)}{2}&\delta(t)Q_{n-1}(t) = \frac{1}{2}\delta(t)\,Q_{n-1}(t)\,+\\
				&\frac{(-1+\phib)}{4}\alpha(t)^2\,Q_{n-2}(t)-\frac{(-1+\phib)}{4}\delta(t)^2\,Q_{n-2}(t),	
			\end{align*}
			
			\noindent and finally,
			
			$$Q_n(t) = (-1+\phib)\frac{\alpha(t)^2-\delta(t)^2}{4}\,Q_{n-2}(t) + \frac{\phib}{2}\delta(t)\,Q_{n-1}(t). $$
			
		\end{proof}

\end{document}